%% file: paper2.tex
\documentclass[11pt, oneside]{amsart}   	
\usepackage{geometry}                		
\usepackage{graphicx}				
\usepackage{cancel}

\usepackage{float}
\usepackage{xcolor}
\usepackage{tikz}
\usetikzlibrary{calc}
\usepackage{caption}
\usepackage{booktabs}
\usepackage{varwidth}

\usepackage{amsthm}
\theoremstyle{plain}
\newtheorem{theorem}{Theorem}[section]   

\newtheorem{lemma}[theorem]{Lemma}
\newtheorem{corollary}[theorem]{Corollary}

\newtheorem{prop}[theorem]{Proposition}   

\theoremstyle{definition}                 
\newtheorem{definition}[theorem]{Definition}

\newtheorem{remark}[theorem]{Remark}      

\numberwithin{equation}{section}

\usepackage{amsmath}
\usepackage{mathtools}
\usepackage{amssymb}

\usepackage[foot]{amsaddr}
\usepackage{commath}
\usepackage{bigints}
\usepackage{romanbar}
\usepackage[ruled,vlined]{algorithm2e}
\usepackage{enumitem}
\setlist{leftmargin=*}

\usepackage{hyperref}
\title[Unconstrained Gradient Flow]{Unconstrained Scheme for Geometrically Constrained Gradient Flows}
\author{S\"oren Bartels}\thanks{Department of Applied Mathematics, University of Freiburg, Hermann--Herder--Str.\ 10, 79104 Freiburg, Germany. \textit{Email:} \texttt{bartels@mathematik.uni-freiburg.de}}

\author{Lucas Bouck}\thanks{Department of Mathematical Sciences, Carnegie Mellon University, 5000 Forbes Ave, Pittsburgh, PA 15213, USA. \textit{Email:} \texttt{lbouck@andrew.cmu.edu}}

\author{Christian Palus}\thanks{Department of Applied Mathematics, University of Freiburg, Hermann--Herder--Str.\ 10, 79104 Freiburg, Germany. \textit{Email:} \texttt{christian.palus@mathematik.uni-freiburg.de}}

\begin{document}

\begin{abstract}
In this paper, we study the approximation of gradient flows of harmonic maps, which serve as model problems for applications in micromagnetics, liquid crystals, and nonlinear plate bending. Harmonic maps are vector fields that are critical points of the Dirichlet energy subject to the constraint that the vector field be unit length pointwise. Most existing time-stepping schemes for gradient flows deal with the constraint by linearizing the unit length constraint at every step, which involves solving for the solution increment in the tangent space of the constraint. These schemes lead to robust control over the violation of the constraint, but require solving degenerate saddle point systems at every step that may be difficult to precondition.  In this paper, we propose a scheme that first computes the unconstrained increment and then projects this increment pointwise onto the tangent space. With an additional stabilization, this scheme is energy stable under mild step size restrictions and provides robust control of the unit length constraint violation. Our new scheme only requires the solution of decoupled symmetric positive definite systems at every step, which translates to a large increase in computational efficiency. We also propose a computable a posteriori criterion and a variable time-stepping procedure that guarantee the stability of the scheme. We conclude with computational examples demonstrating the efficacy of the scheme, and present a computational extension of the scheme to nonlinear plate bending.
\end{abstract}

\maketitle

\section{Introduction}
Minimization problems and gradient flows with geometric or other nonlinear pointwise constraints arise in numerous settings such as liquid crystals \cite{frank1958liquid}, micromagnetics \cite{gilbert2004phenomenological}, and plate bending \cite{friesecke2002theorem}. Beginning with Alouges \cite{alouges1997new}, a successful approach to discretizing such equations linearizes the geometric constraint at every step of the gradient flow by restricting the discrete velocity to live in the tangent space to the constraint. After this step, the predicted solution no longer satisfies the nonlinear geometric constraint, and the algorithm in \cite{alouges1997new} projects the solution back onto the constraint. This projection step is stable at the continuous level for the Dirichlet energy, but it may lead to a lack of energy stability when discretizing with finite elements as explored in the work by the first author in \cite{bartels2005stability}. Additionally, it was proved in \cite{bartels2005stability} that a weakly acute mesh is a sufficient condition for energy stability of the nonlinear projection step.

To bypass energy stability issues with the nonlinear projection of \cite{alouges1997new}, the work of the first author in \cite{bartels2016projection} proposed foregoing the nonlinear projection step entirely, hence the name ``projection-free scheme.'' While the solution no longer satisfies the geometric constraint exactly, the energy stability of the scheme allows \cite{bartels2016projection} to prove an error estimate of the constraint violation in terms of the step size. When the solution is smoother, one can prove error estimates of the projection-free scheme \cite{bartels2024error}. Projection-free schemes have also been extended to energy minimization methods in liquid crystals \cite{nochetto2022gamma, bouck2024projection}, bending isometries \cite{bartels2013bending}, bilayer plates \cite{bartels2022stable}, and accelerated schemes for such problems \cite{dong2025accelerated}.

One drawback of the current suite of projection-free methods is that the linearized system at every step is a saddle point problem. The inf-sup stability of this saddle point problem \cite{bartels2022quasi,hu2009saddle} may depend on higher regularity of the solution and can degenerate in the presence of defects in the solution. While preconditioners for these types of saddle point problems have been developed for liquid crystals \cite{xia2021augmented} and micromagnetics \cite{kraus2019iterative}, preconditioning fourth order problems in plate bending remains a challenge.

 In this paper, we propose a method that only requires solving symmetric positive definite problems at every step. The key idea from the PhD dissertation of the third author \cite[Chapter 5]{palus2024finite} is to split the scheme into two steps. The first step is an unconstrained backward Euler step with an explicit treatment of some nonlinearities depending on the projection operator onto the linear tangent space. The second step is to project the velocity onto the tangent space with an inexpensive pointwise operation. This scheme enables us to avoid computing vectorial saddle point problems and compute only with symmetric positive definite problems. In our computational examples, we show an approximately 4$\times$ to 7$\times$ speed-up over the usual projection-free algorithm for computing harmonic map heat flow and a speed-up of more than $13\times$ in computing critical points of a plate bending problem with an $H^2$ gradient flow compared to previous schemes. The main contributions of this paper are: (i) a complete stability and convergence analysis of the method for harmonic map heat flow and computation of harmonic maps; and (ii) a computational extension to a plate bending problem. To prove stability of our new scheme for harmonic map heat flow, we add a stabilization to the scheme inspired by augmented Lagrangian preconditioners for the Oseen--Frank problem \cite{xia2021augmented}.

 The outline of this paper goes as follows. In Subsections \ref{sec:intro}--\ref{sec:intro-unconstrained-gf-min}, we give an overview of the problem setting, review the projection-free scheme of \cite{bartels2016projection}, and outline the main ideas of the new scheme studied in this paper. In Section \ref{sec:prelim}, we review necessary preliminaries. In Section \ref{sec:harmonic-maps} and Section \ref{sec:harmonic-map-heat-flow}, we prove convergence of our scheme to harmonic maps and convergence of our scheme to solutions of harmonic map heat flow, respectively. Finally, we conclude with computational examples in Section \ref{sec:computations} and a demonstration of our algorithm applied to plate bending problems in Section \ref{sec:plate-bending}.

 Before diving into a more detailed introduction of the method, we must introduce some notation. For a bounded domain $\Omega\subset \mathbb{R}^\ell$, with functions $u,v\in L^2(\Omega; \mathbb{R}^d)$, we denote the $L^2(\Omega; \mathbb{R}^d)$ inner product of $u,v$ by $(u,v)$ and the $L^2(\Omega; \mathbb{R}^d)$ norm of $u$ by $\Vert u\Vert = \sqrt{(u,u)}$. When we need to specify other norms such as other Lebesgue or Sobolev norms, we will abuse notation and only use $L^p$ as the subscript to denote the space when the domain of integration is clearly $\Omega$, i.e.\ $\Vert u\Vert_{L^p}$. If we need to consider a Lebesgue norm whose integral is over a subdomain $S\subset \Omega$, we will specify the domain, i.e.\ $\Vert u\Vert_{L^p(S)}$. Another abuse of notation we will employ throughout this paper is to not specify the target space in the notation for a function space unless necessary. For example, if we have $u:\Omega\to \mathbb{R}^d$, we will use $L^p(\Omega)$ in place of $L^p(\Omega; \mathbb{R}^d)$ for $u$ and $L^p(\Omega)$ in place of $L^p(\Omega; \mathbb{R}^{d\times \ell})$ for $\nabla u$. When we study a time dependent problem on time interval $(0,T)$, we will denote Bochner spaces by $L^p((0,T); W^{1,q}(\Omega))$. The letter $C$ throughout this paper will be a generic constant that will be independent of meshsize $h$, and $C$ may change from line to line. Finally, we also use $a\lesssim b$ to mean that there is a $C>0$ such that $a\leq C b$.
\subsection{Harmonic maps and harmonic map heat flow} \label{sec:intro}
This paper studies the computation of critical points and gradient flows of the Dirichlet energy on the bounded domain $\Omega \subset \mathbb{R}^\ell$:
$$
E[u] = \int_\Omega |\nabla u|^2 dx
$$
subject to the constraint
$$
u\in \mathcal{A} := \{ u\in H^1(\Omega): u =g \text{ on } \Gamma_\mathrm{D} \text{ and } |u|^2 = 1 \text{ a.e.}\}.
$$
The Euler--Lagrange equation of $E$ over the admissible set $\mathcal{A}$ is
\begin{equation}\label{eq:harmonic-maps}
(\nabla u,\nabla w) = 0
\end{equation}
for all $w\in T\mathcal{A}(u) := \{ v\in H^1_\mathrm{D}(\Omega):  v\cdot u = 0 \text{ a.e.}\}$, where $H^1_\mathrm{D}(\Omega)$ is the space of $H^1$ functions whose trace vanishes on $\Gamma_\mathrm{D}$.
The solutions to \eqref{eq:harmonic-maps} are known as {\it harmonic maps}. To compute critical points of $E$, we use gradient flows of $E$. Let $(\cdot,\cdot)_*$ be an inner product on the space $H^1_\mathrm{D}(\Omega)$ where there is a $C>0$ such that the norm induced by $(\cdot,\cdot)_*$ satisfies $\Vert w\Vert \leq C \Vert w\Vert_*$ for all $w\in H^1_\mathrm{D}(\Omega)$. Formally, the $(\cdot,\cdot)_*$ gradient flow of $E$ is to find a $u\in L^\infty((0,T); H^1(\Omega))\cap H^1((0,T); L^2(\Omega))$ that satisfies
\begin{equation}\label{eq:star-gradient-flow}
(\dot{u}, w)_* + (\nabla u, \nabla w) = 0 \text{ a.e.\ in } (0,T),
\end{equation}
for all test functions $w\in L^\infty((0,T); H^1_{\mathrm{D}}(\Omega))$ that satisfy $w(t) \in T\mathcal{A}(u(t))$ a.e.\ $t\in (0,T)$. If $(\cdot,\cdot)_*$ is the $L^2$ inner product, then we have {\it harmonic map heat flow}
\begin{equation} \label{eq:intro-heat-flow}
(\dot{u}, w) + (\nabla u, \nabla w) = 0 \text{ a.e.\ in } (0,T).
\end{equation}

\subsection{Projection-free scheme for harmonic map gradient flow}\label{sec:intro-proj-free}
The projection-free scheme \cite{bartels2016projection} discretizes the harmonic map gradient flow by linearizing the constraint $|u|^2 = 1$ at every step.
Given $u^k$ and time step $\tau_k$, the scheme computes the discrete velocity $d_t u^{k+1}:= \frac{u^{k+1} - u^k}{\tau_k}$ to satisfy $d_t u^{k+1} \in T\mathcal{A}(u^k)$ and
$$
(d_t u^{k+1}, w)_*+ \tau_k (\nabla  d_t u^{k+1}, \nabla w) = -(\nabla u^k, \nabla w)
$$
for all $w\in T\mathcal{A}(u^k)$. In order to practically solve this problem, one would either need to build a basis for the tangent space $T\mathcal{A}(u^k)$ or introduce a Lagrange multiplier $\lambda^{k+1}$ such that $(d_tu^{k+1}, \lambda^{k+1})$ solves
\begin{equation}\label{eq:saddle-point-problem}
\begin{aligned}
(d_t u^{k+1}, w)_*+ \tau_k (\nabla  d_t u^{k+1}, \nabla w) + (\lambda^{k+1}, u^k\cdot w) &= -(\nabla u^k, \nabla w)\\
(d_tu^{k+1}\cdot u^k, \rho) &= 0
\end{aligned}
\end{equation}
for all $(w,\rho)\in H^1_\mathrm{D}(\Omega;\mathbb{R}^d)\times L^2(\Omega)$. Solving the above vectorial saddle point problem can become difficult, especially in 3 spatial dimensions. The difficulty of these saddle point problems becomes more pronounced for fourth order plate bending problems.
\subsection{Unconstrained gradient flow for harmonic map heat flow} \label{sec:intro-heat-flow}
In order to avoid the potentially degenerate saddle point structure of the projection-free scheme, we propose a new scheme that splits the computation of the velocity and its projection into two steps. Instead of testing with functions $w\in T\mathcal{A}(u)$, we test the harmonic map heat flow with functions of the form $P_u w$, where $P_u = I - \tilde{u}\otimes \tilde{u}$, with $\tilde{u} = u/|u|$, is the projection of a vector onto the tangent space $T\mathcal{A}(u)$. The equivalent weak form of harmonic map heat flow from \eqref{eq:intro-heat-flow} is
\begin{equation*} 
(\dot{u}, P_u w) + (\nabla u, \nabla P_u w) = 0 \text{ a.e. $t\in (0,T)$}
\end{equation*}
for all $w\in H^1_\mathrm{D}(\Omega)$.
By expanding $\nabla P_u w$ and using $(\dot{u}, P_u w) = (\dot{u}, w)$ for the $L^2$ flow, the new formulation looks like a heat equation with an additional nonlinearity:
$$
(\dot{u}, w) + (\nabla u, \nabla w) = (\nabla u, \nabla (I - P_u)w).
$$
The idea from \cite{palus2024finite} we explore in this paper uses the above formulation and splits the harmonic map heat flow into two steps. The first step is an unconstrained problem to compute a predicted velocity $v^{k+1}$ with an implicit treatment of the linear part and explicit treatment of the nonlinearity:
\begin{equation}\label{eq:v-eq}
(v^{k+1}, w)+ \tau_k (\nabla  v^{k+1}, \nabla w) = -(\nabla u^k, \nabla w) +(\nabla u^k, \nabla (I - P_{u^k})w) = -(\nabla u^k, \nabla P_{u^k}w).
\end{equation}
The predicted velocity $v^{k+1}$ from \eqref{eq:v-eq} may not be an element of the tangent space $T\mathcal{A}(u^k)$ due to the explicit treatment of the nonlinearity. The second step of the scheme remedies this error and projects $v^{k+1}$ onto the tangent space
\begin{equation}\label{eq:projection}
d_t u^{k+1} = P_{u^k}v^{k+1}.
\end{equation}

The advantage of this scheme is that instead of solving a vectorial saddle point problem in \eqref{eq:saddle-point-problem}, we now only need to solve decoupled scalar SPD problems in \eqref{eq:v-eq}. The ease of solving \eqref{eq:v-eq} comes at the moderate cost of a stability condition.

The stability of this scheme is closely related to the stability of the tangent space projection $P_{{u}^k}$. An initial estimate follows from testing \eqref{eq:v-eq} with $v^{k+1}$ and using \eqref{eq:projection}:
$$
\Vert v^{k+1}\Vert^2  +\tau_k \Vert \nabla  v^{k+1}\Vert^2 =-(\nabla u^k, \nabla P_{{u}^k} v^{k+1}) = -(\nabla u^k, \nabla d_t u^{k+1}).
$$
A simple quadratic identity $(a, a-b) = \frac{1}{2} \Vert a\Vert^2 - \frac{1}{2} \Vert b\Vert^2 + \frac{1}{2} \Vert a-b\Vert^2$ with the above relation yields:
\begin{align*}
\Vert v^{k+1}\Vert^2  +\tau_k \Vert \nabla  v^{k+1}\Vert^2 &=  \frac{1}{\tau_k}(\nabla u^k, \nabla u^k -\nabla u^{k+1} ) 
 = -\frac{1}{2} d_t  \Vert \nabla u^{k+1}\Vert^2 + \frac{\tau_k}{2}\Vert \nabla d_tu^{k+1}\Vert^2
\end{align*}
Rearranging the above equality yields the typical energy identity for the heat flow
\begin{equation}\label{eq:energy-identity-heat}
\frac{1}{2} d_t  \Vert \nabla u^{k+1}\Vert^2+ \Vert v^{k+1}\Vert^2   + \tau_k \Vert \nabla  v^{k+1}\Vert^2 =  \frac{\tau_k}{2}  \Vert \nabla  d_t u^{k+1}\Vert^2.
\end{equation}
Recall from \eqref{eq:projection} that  $d_t u^{k+1} = P_{{u}^k} v^{k+1}$. If the projection $P_{u^k}$ is stable in the sense that
\begin{equation}\label{eq:operator-stability}
\frac{1}{2} \Vert \nabla P_{{u}^k} v^{k+1}\Vert^2\leq \frac{1}{ 2\tau_k} \Vert v^{k+1}\Vert^2 + \Vert \nabla  v^{k+1}\Vert^2,
\end{equation}
then an energy stability bound follows:
\begin{equation}\label{eq:basic-energy-stab-intro}
\frac{1}{2} d_t  \Vert \nabla u^{k+1}\Vert^2 +\frac{1}{2} \Vert v^{k+1}\Vert^2\leq 0 .
\end{equation}
The factor $1/2$ in \eqref{eq:basic-energy-stab-intro} is not quite the correct factor to get the right energy stability bound for harmonic map heat flow, but one can replace $1/2$ with $1/2 - \varepsilon$ resulting in a different stability criterion \eqref{eq:operator-stability}. 

From the energy stability in \eqref{eq:basic-energy-stab-intro} and the orthogonality relation $d_tu^{k+1}\cdot u^k = 0$, the unit length constraint violation is controlled:
\begin{align*}
\Vert |u^{k+1}|^2 - 1\Vert_{L^1(\Omega)} &= \Vert |u^{k}|^2 - 1\Vert_{L^1(\Omega)} + \tau_k^2 \Vert d_tu^{k+1} \Vert^2 \leq \Vert |u^{k}|^2 - 1\Vert_{L^1(\Omega)} + C \tau_k^2 \Vert v^{k+1} \Vert^2\\
&\leq \Vert |u^{k}|^2 - 1\Vert_{L^1(\Omega)} + C \tau_k \left(\Vert \nabla u^{k}\Vert^2 - \Vert \nabla u^{k+1}\Vert^2\right).
\end{align*}
The above inequality telescopes when summing $k=0,\ldots, K$, and if $|u^0| = 1$ a.e., we have the usual control of the unit length constraint violation
$$
\Vert |u^{K+1}|^2 - 1\Vert_{L^1(\Omega)}\leq C\, (\sup_{0\leq k\leq K} \tau_k )\, \Vert \nabla u^{0}\Vert^2.
$$

The key here is that the stability constant of the projection operator $P_{{u}^k}$ is small enough so that \eqref{eq:operator-stability} is valid. When the scheme is implemented using finite element methods, \eqref{eq:operator-stability} is satisfied for $\tau_k$ sufficiently small. The condition on $\tau_k$ is computable and amenable to analysis using inverse inequalities. Indeed, if $v_h^{k+1}, P_{u^k_h}v_h^{k+1}$ are finite element functions, one can use a global inverse inequality and the fact that $P_{u^k_h}$ is a projection to show
$$
\frac{1}{2} \Vert \nabla P_{u^k_h} v_h^{k+1}\Vert^2 \leq Ch^{-2} \Vert P_{u^k_h} v_h^{k+1}\Vert^2 \leq Ch^{-2} \Vert v_h^{k+1}\Vert^2.
$$
From the above inequality, a sufficient condition to satisfy \eqref{eq:operator-stability} is
\begin{equation*}
\frac{C}{2h^2} \Vert v_h\Vert^2\leq \frac{1}{ 2\tau_k} \Vert v_h\Vert^2,
\end{equation*}
which would imply a restrictive stability condition on par with an explicit method: $\tau_k \leq C h^{2}$. To fix this issue, we insert the stabilization $(\tilde{u}^k\cdot v^{k+1},\tilde{u}^k\cdot w)$ with a penalty parameter inside \eqref{eq:v-eq} to write a new scheme
\begin{equation}\label{eq:v-eq-stab}
(v^{k+1}, w)+ \tau_k (\nabla  v^{k+1}, \nabla w) +\gamma (\tilde{u}^k\cdot v^{k+1},\tilde{u}^k\cdot w)=  -(\nabla u^k, \nabla P_{u^k}w).
\end{equation}
In the limit as $\tau_k\to0$, we expect that $\tilde{u}^k\cdot v^{k+1}\to 0$, but when $\tilde{u}^k\cdot v^{k+1}\neq0$, the stabilization creates a stronger norm to improve the coercivity of the problem and control the projection $P_{u^k}$. These types of stabilizations are not new and have been previously used in the context of augmented Lagrangian preconditioners for the Oseen--Frank problem \cite{xia2021augmented}.
To get an improved stability estimate with the modification, test with $v^{k+1}$ in \eqref{eq:v-eq-stab} and get an identity similar to \eqref{eq:energy-identity-heat}:
$$
\frac{1}{2} d_t  \Vert \nabla u^{k+1}\Vert^2  + \gamma\Vert \tilde{u}^k\cdot v^{k+1}\Vert^2 + \Vert v^{k+1}\Vert^2   + \tau_k \Vert \nabla  v^{k+1}\Vert^2 = \frac{\tau_k}{2}  \Vert \nabla  d_t u^{k+1}\Vert^2
$$
The above energy identity with $\gamma>0$ leads to a new stability criterion
\begin{equation}\label{eq:operator-stability-stab}
\frac{1}{2} \Vert \nabla P_{{u}^k} v^{k+1}\Vert^2\leq \frac{\gamma}{ \tau_k} \Vert \tilde{u}^k\cdot v^{k+1}\Vert^2+ \Vert \nabla  v^{k+1}\Vert^2.
\end{equation}
If  \eqref{eq:operator-stability-stab} is satisfied, there is an improved energy estimate with the correct dissipation
$$
\frac{1}{2} d_t  \Vert \nabla u^{k+1}\Vert^2 + \Vert v^{k+1}\Vert^2 \leq 0 .
$$

In a finite element implementation, scaling $\gamma$ with a negative power of $h$ leads to improved stability conditions on $\tau_k$. Splitting $P_{{u}_h^k} v_h^{k+1} = v_h^{k+1} - \tilde{u}_h^k (\tilde{u}_h^k\cdot v_h^{k+1})$ and using a global inverse inequality yields
\begin{align*}
\frac{1}{2} \Vert \nabla P_{{u}^k_h} v^{k+1}_h\Vert^2 \leq \Vert \nabla  v^{k+1}_h\Vert^2 + \Vert \nabla (\tilde{u}_h^k (\tilde{u}_h^k\cdot v_h^{k+1}))\Vert^2\leq \Vert \nabla  v^{k+1}_h\Vert^2 + C h^{-2} \Vert  \tilde{u}_h^k\cdot v_h^{k+1}\Vert^2.
\end{align*}
Hence, \eqref{eq:operator-stability-stab} is satisfied when
$
\tau_k\leq C \gamma h^2,
$
and scaling $\gamma\approx h^{-1}$ leads to a mild stability condition $\tau_k \leq C h$, at the expense of coupled components in the linear solve.

\subsection{Unconstrained gradient flows for energy minimization} \label{sec:intro-unconstrained-gf-min}
While this scheme is best justified for $L^2$ flows, we can apply this idea to design energy stable schemes in more general Hilbert spaces for the purposes of computing harmonic maps. For gradient flows with the inner product $(\cdot, \cdot)_*$, we propose the same scheme as in the $L^2$ case from \eqref{eq:v-eq}:
\begin{equation}\label{eq:star-flow-intro}
(v^{k+1}, w)_*+ \tau_k (\nabla  v^{k+1}, \nabla w) = -(\nabla u^k, \nabla P_{u^k}w),
\end{equation}
which may also be interpreted as a fixed point iteration for computing the optimality condition \eqref{eq:harmonic-maps}.

Testing with $v^{k+1}$ and repeating the arguments from the harmonic map heat flow leads to an analogous identity to \eqref{eq:energy-identity-heat}
$$
\frac{1}{2} d_t  \Vert \nabla u^{k+1}\Vert^2 = - \Vert v^{k+1}\Vert_*^2   - \tau_k \Vert \nabla  v^{k+1}\Vert^2 + \frac{\tau_k}{2}  \Vert \nabla  P_{u^k} v^{k+1}\Vert^2.
$$
If we require
\begin{equation}\label{eq:operator-stability-star-flow}
\frac{1}{2} \Vert \nabla P_{{u}^k} v^{k+1}\Vert^2\leq \frac{1}{ 2\tau_k} \Vert v^{k+1}\Vert_*^2 + \Vert \nabla  v^{k+1}\Vert^2,
\end{equation}
then the scheme satisfies an energy bound of the form
$$
\frac{1}{2} d_t  \Vert \nabla u^{k+1}\Vert^2 + \frac{1}{2} \Vert v^{k+1}\Vert_*^2 \leq 0.
$$
By choosing a stronger norm for the gradient flow, we can achieve a mild stability condition on the time step $\tau_k$ with no stabilization, i.e.\ we can set $\gamma=0$. An advantage of having no stabilization is that \eqref{eq:star-flow-intro} can be solved by solving $d$ decoupled scalar SPD problems, which further enhances the efficiency of the scheme.

To see a stability condition on $\tau_k$ in a finite element implementation, consider the case of $(v,w)_* = (\nabla v, \nabla w)$. Applying the definition $P_u = I - \tilde{u}\otimes\tilde{u}$, product rule $\partial_i(\tilde{u}_h^k (\tilde{u}_h^k\cdot v_h^{k+1})) = \partial_i\tilde{u}_h^k (\tilde{u}_h^k\cdot v_h^{k+1})+ \tilde{u}_h^k(\partial_i\tilde{u}_h^k\cdot v_h^{k+1} +\tilde{u}_h^k\cdot \partial_iv_h^{k+1})$, and H\"older's inequality leads to a bound on the LHS of \eqref{eq:operator-stability-star-flow}:
$$
 \Vert \nabla P_{{u}^k} v_h^{k+1}\Vert \leq \Vert \nabla v_h^{k+1}\Vert + 2\Vert \nabla \tilde{u}_h^k\Vert \Vert \tilde{u}_h^k\Vert_{L^\infty} \Vert v_h^{k+1}\Vert_{L^\infty} + \Vert \tilde{u}_h^k\Vert_{L^\infty}^2 \Vert \nabla v_h^{k+1}\Vert.
$$
To complete the argument, we introduce $\rho_{\mathrm{inv}} = C|\log h|^{1/2}$ for $\ell=2$ and $\rho_{\mathrm{inv}} = Ch^{1-\ell/2}$ for $\ell \geq3$. The global inverse inequality to control the $L^\infty$ norm with the $H^1$ norm is $\Vert v_h^{k+1}\Vert_{L^\infty}\leq  \rho_{\mathrm{inv}} \Vert \nabla v_h^{k+1}\Vert$. This inverse inequality and $|\tilde{u}_h^k| = 1$ lead to a bound with an explicit $h$ dependence
$$
 \Vert \nabla P_{{u}^k} v_h^{k+1}\Vert \leq  C\left( 1+ \rho_{\mathrm{inv}}  \Vert \nabla \tilde{u}_h^k\Vert\right) \Vert \nabla v_h^{k+1}\Vert.
$$
Hence, the stability bound \eqref{eq:operator-stability-star-flow} is satisfied when $\tau_k \leq C \rho_{\mathrm{inv}}^{-2}$. In two spatial dimensions, this leads to a favorable stability condition $\tau_k \leq C |\log h|^{-1}$, and in three dimensions, the stability condition becomes $\tau_k \leq C h$.

\section{Preliminaries}\label{sec:prelim}

Let $\Omega \subset \mathbb{R}^\ell$ be a bounded domain fitted by a shape-regular sequence of meshes $\{\mathcal{T}_h\}_{h>0}$. Each element $T\in \mathcal{T}_h$ will have a local mesh size $h_T$, and we write $h = \max\{ h_T : T\in \mathcal{T}_h\}$ and $h_{\mathrm{min}} = \min\{ h_T : T\in \mathcal{T}_h\}$ for the maximum and minimum element diameters. In this paper, we consider the standard continuous piecewise affine finite element space
\begin{equation}\label{eq:fe-space}
\mathcal{S}^{1}(\mathcal{T}_{h} ; \mathbb{R}^{d})=\{w_{h} \in C(\bar{\Omega} ; \mathbb{R}^{d}):\left.w_{h}\right|_{T}  \text{ is affine for all } T \in \mathcal{T}_{h}\},
\end{equation}
as well as the usual Lagrange nodal interpolant denoted by $\mathcal{I}_h$. Additionally, we will denote the discrete space with homogeneous Dirichlet boundary conditions on $\Gamma_D$ by $\mathcal{S}_\mathrm{D}^{1}(\mathcal{T}_{h} ; \mathbb{R}^{d})$.

In order to discretize the unit length constraint $|u|^2 = 1$ a.e., we enforce a relaxed unit length constraint at nodes $z\in \mathcal{N}_h$ with tolerance $\delta>0$. The Dirichlet boundary condition $u=g$ is enforced on $\Gamma_\mathrm{D} \neq\emptyset$. In order to enforce boundary conditions for the discrete problem at nodes, we assume $g$ is the trace of a function in $W^{1,p}(\Omega)$ for $p>\ell$. The admissible set for the discrete problem is
\begin{equation}\label{eq:discrete-admissible-set}
\begin{aligned} \mathcal{A}_{h,\delta}=\{w_{h} \in \mathcal{S}^{1}(\mathcal{T}_{h} ; \mathbb{R}^{d}):  w_{h}=g \text { in }\mathcal{N}_{h} \cap \Gamma_{\mathrm{D}},\, \Vert \mathcal{I}_h [|w_h|^2 - 1]\Vert_{L^1(\Omega)} \leq \delta \}\end{aligned}.
\end{equation}
A standard result in \cite[Example 4.6]{bartels2015numerical} is that the Dirichlet energy over the discrete admissible set $\mathcal{A}_{h,\delta}$ $\Gamma$-converges in the weak $H^1(\Omega)$ topology to the Dirichlet energy over the continuous admissible set $\mathcal{A}$ as $h,\delta\to0$.

An important space for both the unconstrained and the constrained flow is the tangent space to $\mathcal{A}_{h,\delta}$ at $u_h$. This space is defined by
\begin{equation}\label{eq:discrete-tangent-space}
T\mathcal{A}_{h,\mathrm{D}}(u_h) := \{ w_{h} \in \mathcal{S}^{1}_D(\mathcal{T}_{h} ; \mathbb{R}^{d}) : w_{h}(z)\cdot u_h(z) = 0 \text { for all } z \in \mathcal{N}_{h}\}.
\end{equation}
\subsection{Stability of pointwise tangent space projection}
Recall from \eqref{eq:operator-stability} that a key ingredient of the stability of the unconstrained scheme is the stability of the tangent space projection operator. Defining the discrete tangent space projection requires a discrete version of $u\mapsto u/|u|$. Our definition  follows \cite{bartels2005stability, bartels2016projection}.
\begin{definition}[discrete nodal projection]
Let $u_h\in \mathcal{S}^{1}(\mathcal{T}_{h} ; \mathbb{R}^{d})$ be such that $|u_h(z)| \geq 1$ for all $z\in \mathcal{N}_h$. We define its nodal projection $\tilde{u}_h$ as $\tilde{u}_h = \mathcal{I}_h[u_h / |u_h|]$.
\end{definition}
The tangent space projection operator is defined similarly.
\begin{definition}[discrete tangent space projection]\label{def:discrete-tangent-proj}
Let $u_h \in \mathcal{A}_{h,\delta}$ with $|u_h(z)| \geq 1$ for all $z\in \mathcal{N}_h$. The discrete tangent projection $P_{h,u_h}:\mathcal{S}^{1}_{\mathrm{D}}(\mathcal{T}_{h} ; \mathbb{R}^{d}) \to T\mathcal{A}_{h,\mathrm{D}}(u_h)$ is defined by
\begin{equation}\label{eq:def-proj}
P_{h,u_h} w_h = w_h - \mathcal{I}_h \left( \tilde{u}_h(\tilde{u}_h\cdot w_h)\right).
\end{equation}
\end{definition}
The stability of the time-stepping scheme relies on estimates on the operator norms of $P_{h,u_h}$. The lemma below estimates the operator norm using $\Vert \nabla \tilde{u}_h\Vert_{L^\infty} $, $\Vert  \nabla u_h \Vert $, and $h_{\mathrm{min}}$. As in \cite{bartels2022stable, bouck2024projection}, the proof uses a discrete Sobolev inequality $\Vert w_h\Vert_{L^\infty} \leq C\rho_{\mathrm{inv}} \Vert \nabla w_h\Vert$ (see \cite{brenner2008mathematical} and \cite[Remark 3.8]{bartels2015numerical}).
\begin{lemma}[stability of $P_{h,u_h}$]\label{lem:proj-stab}
Suppose $|u_h(z)| \geq 1$ for all $z\in \mathcal{N}_h$. Let $P_{h,u_h}$ be the operator defined in Definition \ref{def:discrete-tangent-proj}. There is a constant $C>0$ independent of $h, h_{\mathrm{min}}$ such that the following stability bounds on $P_{h,u_h}$ hold for all $w_h\in \mathcal{S}_\mathrm{D}^1(\mathcal{T}_h;\mathbb{R}^{d})$:
\begin{align}
\Vert P_{h,u_h} w_h\Vert &\leq C\Vert w_h\Vert, \label{eq:proj-l2-l2-stab}\\
\Vert \nabla P_{h,u_h} w_h\Vert &\leq  \Vert \nabla w_h\Vert +Ch_{\mathrm{min}}^{-1}\Vert \mathcal{I}_h (w_h\cdot \tilde{u}_h)\Vert ,\label{eq:proj-l2-h1-stab}\\
\Vert \nabla P_{h,u_h} w_h\Vert &\leq  C\big(1 +  h \, (1+ \rho_{\mathrm{inv}}\|\nabla {u}_{h}\| )   \|\nabla \tilde{u}_{h}\|_{L^{\infty}}+\rho_{\mathrm{inv}} \Vert \nabla {u}_h\Vert\big)   \Vert \nabla w_h\Vert, \label{eq:proj-h1-h1-stab}
\end{align}
where $\rho_{\mathrm{inv}} = h_{\mathrm{min}}^{1-\frac{\ell}{2}}$ for  $\ell \geq 3$ and $|\log h_{\mathrm{min}}|^{1/2}$ for $\ell=2$.
\end{lemma}
\begin{proof}
We prove the three bounds one-by-one.

\noindent\textit{Step 1. Proof of \eqref{eq:proj-l2-l2-stab}:} At every node $z\in \mathcal{N}_h$, $|(I - \tilde{u}_h(z) \tilde{u}_h(z)^\top)w_h(z)| \leq |w_h(z)|$. Combining the nodal bound with the norm equivalence $\Vert v_h\Vert \approx (\sum_{z\in \mathcal{N}_h} h_z^\ell |v_h(z)|^2)^{1/2}$, where $h_z$ is the diameter of the local patch at node $z$, proves \eqref{eq:proj-l2-l2-stab}.

\noindent\textit{Step 2. Proof of \eqref{eq:proj-l2-h1-stab}:} The definition $P_{h,u_h} w_h = w_h - \mathcal{I}_h \left( \tilde{u}_h(\tilde{u}_h\cdot w_h)\right)$, the triangle inequality, and the inverse estimate $\Vert \nabla v_h\Vert \leq C h_{\mathrm{min}}^{-1} \Vert v_h\Vert$ prove
$$
\Vert \nabla P_{h,u_h} w_h\Vert \leq \Vert\nabla w_h\Vert + \Vert\nabla \mathcal{I}_h \big( \tilde{u}_h(\tilde{u}_h\cdot w_h)\big) \Vert
\leq\Vert\nabla w_h\Vert  + C h_{\mathrm{min}}^{-1}\Vert \mathcal{I}_h ( \tilde{u}_h(\tilde{u}_h\cdot w_h)) \Vert.
$$
The nodal bound $|\tilde{u}_h(z)(\tilde{u}_h(z)\cdot w_h(z))|\leq |\tilde{u}_h(z)\cdot w_h(z)|$ for all $z\in \mathcal{N}_h$ with the same norm equivalence above proves that the second term on the RHS above satisfies $\Vert \mathcal{I}_h \left( \tilde{u}_h(\tilde{u}_h\cdot w_h)\right) \Vert \leq C\Vert \mathcal{I}_h ( \tilde{u}_h\cdot w_h) \Vert$, which completes the proof of \eqref{eq:proj-l2-h1-stab}.

\noindent\textit{Step 3. Proof of \eqref{eq:proj-h1-h1-stab}:} The triangle inequality and the definition of $P_{h,u_h}$ in \eqref{eq:def-proj} yield:
\begin{align*}
\Vert \nabla P_{h,u_h} w_h\Vert &\leq \Vert \nabla  w_h\Vert +\Vert \nabla\big[  \tilde{u}_h(\tilde{u}_h\cdot w_h) - \mathcal{I}_h \left( \tilde{u}_h(\tilde{u}_h\cdot w_h)\right)\big]\Vert+  \Vert \nabla \left( \tilde{u}_h(\tilde{u}_h\cdot w_h)\right)\Vert\\
& =\Vert \nabla  w_h\Vert + I +  II
\end{align*}
The bound of $I$ follows the arguments from \cite{bartels2016projection}. For each element $T\in\mathcal{T}_h$, we first use a local interpolation error estimate
$$
\Vert \nabla\left[  \tilde{u}_h(\tilde{u}_h\cdot w_h) - \mathcal{I}_h ( \tilde{u}_h(\tilde{u}_h\cdot w_h))\right]\Vert_{L^{2}(T)}
\leq ch_T\|D^{2}(\tilde{u}_{h} \tilde{u}_{h}^{\top} w_{h})\|_{L^{2}(T)}.
$$
The product rule, the affineness of $\tilde{u}_h$ and $w_h$, and H\"older's inequality control the Hessian:
\begin{align*}
&\|D^{2}(\tilde{u}_{h} \tilde{u}_{h}^{\top} w_{h})\|_{L^{2}(T)}  \leq C\|\nabla \tilde{u}_{h}\|_{L^{2}(T)}\|\nabla \tilde{u}_{h}\|_{L^{\infty}(T)}\|w_{h}\|_{L^{\infty}(T)}+C\|\tilde{u}_{h}\|_{L^{\infty}(T)}\|\nabla \tilde{u}_{h}\|_{L^{\infty}(T)}\|\nabla w_{h}\|_{L^{2}(T)}.
\end{align*}
Using the inequality $(a+b)^2 \leq 2 a^2 +2 b^2$ and summing over elements, we have
$$
I \leq C h \|\nabla \tilde{u}_{h}\|_{L^{\infty}} \left(\|\nabla \tilde{u}_{h}\|\|w_{h}\|_{L^{\infty}}+\|\tilde{u}_{h}\|_{L^{\infty}}\|\nabla w_{h}\|\right) \leq C h \|\nabla \tilde{u}_{h}\|_{L^{\infty}} \left(\|\nabla \tilde{u}_{h}\|\|w_{h}\|_{L^{\infty}}+\|\nabla w_{h}\|\right).
$$
The discrete Sobolev inequality 
$
\|w_{h}\|_{L^{\infty}} \leq C\rho_{\mathrm{inv}}\| \nabla w_{h}\|
$
leads to
$$
I \leq C h \, (1+ \rho_{\mathrm{inv}}\|\nabla \tilde{u}_{h}\| )   \|\nabla \tilde{u}_{h}\|_{L^{\infty}}\|\nabla w_{h}\|.
$$
A second use of the product rule, H\"older's inequality, and the $L^\infty$ to $H^1$ discrete Sobolev inequality handles $II$:
$$
II \leq C  \Vert  \tilde{u}_h\Vert_{L^\infty}\Vert \nabla \tilde{u}_h \Vert \Vert w_h\Vert_{L^\infty} + C  \Vert  \tilde{u}_h\Vert_{L^\infty}^2 \Vert \nabla w_h\Vert  \leq C(1+ \rho_{\mathrm{inv}} \Vert \nabla \tilde{u}_h\Vert) \Vert \nabla w_h\Vert.
$$
Inserting the estimates of $I,II$ into the initial bound of $P_{h,u_h}$ yields
$$
\Vert \nabla P_{h,u_h} w_h\Vert \leq C\big(1 +  h \, (1+ \rho_{\mathrm{inv}}\|\nabla \tilde{u}_{h}\| )   \|\nabla \tilde{u}_{h}\|_{L^{\infty}}+  \rho_{\mathrm{inv}} \Vert \nabla \tilde{u}_h\Vert\big)   \Vert \nabla w_h\Vert.
$$
The nodal projection $u_h\mapsto \tilde{u}_h$ is stable in $H^1$ in the sense that there is a potentially large constant $C>0$ depending only on mesh geometry such that $\|\nabla \tilde{u}_{h}\| \leq C\|\nabla {u}_{h}\|$ for all $u_h\in \mathcal{S}^{1}(\mathcal{T}_{h} ; \mathbb{R}^{d})$ with $|u_h(z)| \geq 1$ for all $z\in \mathcal{N}_h$ \cite[Lemma 2.2]{bartels2016projection}. Inserting $\|\nabla \tilde{u}_{h}\| \leq C\|\nabla {u}_{h}\|$ proves the result.
\end{proof}

\section{Harmonic Maps} \label{sec:harmonic-maps}

This section addresses energy minimization of harmonic maps. Given a $\tau>0$ and $u^k_h\in \mathcal{A}_{h,\delta}$, the unconstrained step $v_h^{k+1} \in \mathcal{S}_{\mathrm{D}}^1(\mathcal{T}_h; \mathbb{R}^d)$ will solve
\begin{equation}\label{eq:discrete-flow}
    (v_h^{k+1}, w_h)_* + \gamma (\mathcal{I}_h(\tilde{u}_h^k \cdot v_h^{k+1}), \mathcal{I}_h(\tilde{u}_h^k \cdot w_h)) + \tau(\nabla v_h^{k+1}, \nabla w_h) = -(\nabla u_h^k, \nabla P_{h,u_h^k} w_h),
\end{equation}
for all $w_h \in \mathcal{S}_{\mathrm{D}}^1(\mathcal{T}_h; \mathbb{R}^d)$.

In order to define the variable time-stepping scheme and quantify the stability of the tangent space projection operator $P_{h,u_h}$, we introduce the following ratio of norms to condense the notation:
\begin{equation}\label{eq:norm-ratio}
R_{h,\gamma,*}[v_h^{k+1}] := 2\frac{\|v_h^{k+1}\|_*^2+ \gamma \Vert \mathcal{I}_h(\tilde{u}_h^k \cdot v_h^{k+1})\Vert^2}{ \|\nabla P_{h,u_h^k}v_h^{k+1}\|^2}.
\end{equation}
As we will see later,  $\tau \leq (1-\alpha)R_{h,\gamma,*}[v_h^{k+1}] $ is equivalent to the stability condition \eqref{eq:stab-criteria}, which will imply the monotone decay of the energy.

Before stating the gradient flow algorithm in Alg.~\ref{alg:unconstrained-flow}, we remark on a practical aspect of $R_{h,\gamma,*}[v_h^{k+1}]$. At first glance, the above ratio is not well-defined when $\nabla P_{h,u_h^k}v_h^{k+1}=0$. However, if $v_h^{k+1}$ solves \eqref{eq:discrete-flow} and $\nabla P_{h,u_h^k}v_h^{k+1} = 0$, then we use the test function $w_h = v_h^{k+1}$ in \eqref{eq:discrete-flow} to show that $v_h^{k+1}=0$. In this case, $u^k_h$ would already be a critical point of the Dirichlet energy in the admissible set $\mathcal{A}_{h,\delta}$.

\begin{algorithm}[h]
\SetAlgoLined
\KwIn{initial value $u_h^0 \in \mathcal{A}_{h,0}$, stopping criterion $\varepsilon> 0$, penalty parameter $\gamma\geq 0$, adaptive time step parameter $0< \alpha<1$, initial step size $\tau_0$, and maximum step size $\tau_{\mathrm{max}}$. }

\For{$k=0,\ldots$}{
\textbf{(1)} Compute $v^{k+1}_h$ as in \eqref{eq:discrete-flow}\;
    \eIf{
    $\tau \leq (1-\alpha)R_{h,\gamma,*}[v_h^{k+1}]$}{
	Set $\tau_k = \tau$, $d_tu^{k+1}_h = P_{h,u_h^k}v_h^{k+1}$ and $u_h^{k+1} = u_h^k + \tau_k d_t u_h^{k+1}$\;
        Increase $\tau$ via $\tau \mapsto \min\{\tau_{\mathrm{max}},  (1-\alpha)R_{h,\gamma,*}[v_h^{k+1}]\}$, and continue with (2)\;
    }{Decrease $\tau$ via $\tau \mapsto (1-\alpha)R_{h,\gamma,*}[v_h^{k+1}]$, and go to (1) to repeat computation\;}

\textbf{(2)}\;
  \If{$\|v_h^{k+1}\|_* < \varepsilon$}{
        Stop the iteration and set $u_h^{\infty} = u_h^{k+1}$\;
    }
}

\caption{Unconstrained scheme for energy minimization}
\label{alg:unconstrained-flow}
\end{algorithm}
The time-step update is just one of many possible strategies for adjusting $\tau$. As will be seen in Section \ref{sec:computations}, the increase in $\tau$ is often too large, and the stability bound \eqref{eq:stab-criteria} below is violated in the next step. Alg.~\ref{alg:unconstrained-flow} then decreases $\tau$ and recomputes \eqref{eq:discrete-flow}, which sometimes doubles the cost per iteration compared with the same scheme run with constant step sizes. In this paper, we are primarily interested in the stability and convergence properties of the method, and leave finding a more efficient adaptive time-stepping algorithm to future work. The unconstrained flow algorithm enjoys some energy stability properties as long as a computable a posteriori criterion holds, which we state in the proposition below.

\begin{prop}[a posteriori energy stability for $(\cdot, \cdot)_*$ flows]\label{prop:a-posteriori-energy-stability}
Let $0 < \alpha< 1$, $0\leq\gamma<\infty$, and let $\tau_k$ be a sequence of step sizes chosen such that $\tau_k\leq \tau_{\mathrm{max}}$.
If the scheme satisfies
\begin{equation}\label{eq:stab-criteria}
 \frac{\tau_k}{2} \Vert \nabla P_{h, u^k_h} v_h^{k+1}\Vert^2 \leq (1-\alpha) \left( \Vert v_h^{k+1}\Vert_*^2 + \gamma \Vert \mathcal{I}_h(\tilde{u}_h^k \cdot v_h^{k+1})\Vert^2\right),
\end{equation}
then the scheme satisfies the energy estimate
\begin{equation}\label{eq:star-flow-stability}
\frac{1}{2} \Vert \nabla u^{k+1}_h\Vert^2 + \tau_k \alpha\big(\Vert v^{k+1}_h \Vert_*^2 + \gamma \Vert \mathcal{I}_h(\tilde{u}_h^k \cdot v_h^{k+1})\Vert^2\big) \leq \frac{1}{2}  \Vert \nabla u^{k}_h\Vert^2.
\end{equation}
Moreover, if $(\cdot,\cdot)_*$ controls the $L^2$ norm in the sense that $\Vert w_h\Vert \lesssim \Vert w_h\Vert_*$, and $u_h^0 \in \mathcal{A}_{h,0}$, then there is a constant $C>0$ such that
\begin{equation}\label{eq:unit-length-error-control}
\Vert \mathcal{I}_h [|u^{k+1}_h|^2 - 1]\Vert_{L^1(\Omega)} \leq C\alpha^{-1} \tau_{\mathrm{max}} \Vert \nabla u^0_h\Vert^2.
\end{equation}
\end{prop}
\begin{proof}
First, test \eqref{eq:discrete-flow} with $v^{k+1}_h$, recognize that $P_{h,u_h^k}v_h^{k+1}  = d_tu^{k+1}_h$, and repeat the arguments from \eqref{eq:v-eq} to \eqref{eq:energy-identity-heat} in Section \ref{sec:intro-heat-flow} by replacing $(\cdot,\cdot)$ with $(\cdot,\cdot)_*$. The resulting energy identity is
\begin{equation}\label{eq:energy-identity}
d_t \frac{1}{2} \Vert \nabla u_h^{k+1}\Vert^2 +   \Vert v_h^{k+1}\Vert_*^2 + \gamma \Vert \mathcal{I}_h(\tilde{u}_h^k \cdot v_h^{k+1})\Vert^2 + \tau_k\Vert \nabla v_h^{k+1}\Vert^2  = \frac{\tau_k}{2} \Vert \nabla P_{h, u^k_h} v_h^{k+1}\Vert^2.
\end{equation}
Applying \eqref{eq:stab-criteria} to bound the RHS above and absorbing it into the LHS immediately proves \eqref{eq:star-flow-stability}. The error control \eqref{eq:unit-length-error-control} follows immediately from \eqref{eq:star-flow-stability}, arguments in \cite{bartels2016projection} (also outlined in Section \ref{sec:intro-heat-flow}), and tracking the computations with $\alpha$. 
\end{proof}
In addition to the iteration being stable under an a posteriori condition, Alg.~\ref{alg:unconstrained-flow} satisfies \eqref{eq:stab-criteria} given a small enough $\tau$ even if $\gamma=0$.
\begin{corollary}[conditional stability of $H^1$ flow with constant step sizes]\label{cor:conditional-stability-h1-flow}
Assume $\{\mathcal{T}_h\}_{h>0}$ is a quasiuniform sequence of meshes, $0<\alpha<1$, and $\gamma\geq0$. Let $(\cdot,\cdot)_* = (\cdot,\cdot)_{H^1}$. Finally, assume constant step sizes $\tau_k = \tau$ for all $k$. The sequence $u^k_h$ generated by Alg.~\ref{alg:unconstrained-flow} satisfies the stability estimate \eqref{eq:star-flow-stability} and control over the unit length constraint violation
\eqref{eq:unit-length-error-control} provided
\begin{equation}\label{eq:h1-cfl}
\tau \leq C(1-\alpha)(1+\|\nabla {u}^0_{h}\|^2)^{-1}\rho_{\mathrm{inv}}^{-2}, \quad \rho_\mathrm{inv}^{-2} =  \begin{cases}
h^{\ell - 2}, &\quad \text{ for } \ell \geq 3\\
|\log h|^{-1},& \quad \text{ for } \ell =2
\end{cases},
\end{equation}
where $C$ is a constant independent of $h$, and $\ell$ is the spatial dimension.
\end{corollary}
\begin{proof}
It is sufficient to prove that if $\tau$ is small enough, then the stability criterion \eqref{eq:stab-criteria} is satisfied for all steps via an inductive argument. Employing \eqref{eq:proj-h1-h1-stab} from Lemma \ref{lem:proj-stab} (stability of $P_{h,u_h}$), the discrete Sobolev inequality $\| {u}^k_{h}\|_{L^\infty} \leq C\rho_{\mathrm{inv}}\|\nabla {u}^k_{h}\|$, and the inverse inequality $\| \nabla \tilde{u}^k_{h}\|_{L^\infty} \leq C h^{-1}_{\mathrm{min}} \| \tilde{u}^k_{h}\|_{L^\infty}$, the RHS is bounded as follows:
\begin{align*}
\frac{\tau}{2} \Vert \nabla P_{h, u^k_h} v_h^{k+1}\Vert^2 &\leq \frac{C\tau}{2} \bigg(1 +  h \, (1+ \rho_{\mathrm{inv}}\|\nabla \tilde{u}^k_{h}\| )   \|\nabla \tilde{u}^k_{h}\|_{L^{\infty}}+  \rho_{\mathrm{inv}} \Vert \nabla {u}^k_h\Vert \bigg)^2   \Vert \nabla v_h^{k+1}\Vert^2 \\
&\leq \frac{C\tau}{2} \bigg(1 +  \frac{h}{h_{\mathrm{min}}} \| \tilde{u}^k_{h}\|_{L^{\infty}}(1+ \rho_{\mathrm{inv}}\|\nabla \tilde{u}^k_{h}\| )+  \rho_{\mathrm{inv}} \Vert \nabla {u}^k_h\Vert \bigg)^2   \Vert \nabla v_h^{k+1}\Vert^2\\
&\leq \frac{C\tau}{2} \left(1+ \frac{h^2}{h_{\mathrm{min}}^2}\rho_{\mathrm{inv}}^2(1+\|\nabla {u}^k_{h}\|^2)\right)\Vert \nabla v_h^{k+1}\Vert^2.
\end{align*}
The stability criterion \eqref{eq:stab-criteria} is clearly satisfied for $k=0$ with $1 - \alpha >0$ provided $\{\mathcal{T}_h\}_{h>0}$ is quasiuniform, and $\tau$ is chosen to satisfy \eqref{eq:h1-cfl}. The energy stability of the first step gives $\|\nabla {u}^1_{h}\|^2\leq \|\nabla {u}^0_{h}\|^2$, and $\tau$ is small enough to fulfill \eqref{eq:stab-criteria} for $k=1$. An inductive argument and a repeat of the arguments in the proof of Proposition \ref{prop:a-posteriori-energy-stability} completes the proof.
\end{proof}

\begin{remark}[termination]\label{rmk:termination}
An important consequence of Corollary \ref{cor:conditional-stability-h1-flow} is that there is always a $\tau$ sufficiently small such that \eqref{eq:stab-criteria} is satisfied, or equivalently, $R_{h,\gamma,*}[v_h^{k+1}]$ and $\tau_k$ are always bounded from below. As a result, for $0<\alpha<1$, Alg.~\ref{alg:unconstrained-flow} will always produce a sequence $\tau_k, u^k_h$ such that \eqref{eq:stab-criteria} is satisfied. The lower bound on $\tau_k$ in addition to the telescoping nature of the energy estimate \eqref{eq:star-flow-stability} also guarantees that for all $\varepsilon>0$, there is a $K_\varepsilon$ such that Alg.~\ref{alg:unconstrained-flow} will terminate at $k = K_\varepsilon$.
\end{remark}
\begin{remark}[convergence to harmonic maps]
Let $u_{h, \tau_{\mathrm{max}}, \varepsilon}$ be the output of Alg.~\ref{alg:unconstrained-flow} for $0<\alpha<1$ and $\gamma\geq0$. With some additional hypotheses, one can show that as $h,\tau_{\mathrm{max}},  \varepsilon \to0$, an accumulation point $u$ of $\{u_{h, \tau_{\mathrm{max}}, \varepsilon}\}_{h,\tau_{\mathrm{max}},\varepsilon>0}$ is a harmonic map, i.e.\ $u\in \mathcal{A}$ and $(\nabla u, \nabla w) = 0$ for all $w\in H^1_{\mathrm{D}}(\Omega)\cap L^\infty(\Omega)$ such that $w\cdot u = 0$ a.e. The precise statement and proof follow \cite[Theorem 7.6]{bartels2015numerical}.
\end{remark}

\section{Harmonic Map Heat Flow}\label{sec:harmonic-map-heat-flow}
This section explores a variant of the unconstrained scheme for harmonic map heat flow with penalty parameter $\gamma>0$. Given a $\tau>0$ and $u^k_h\in \mathcal{A}_{h,\delta}$, the unconstrained velocity $v_h^{k+1} \in \mathcal{S}_{\mathrm{D}}^1(\mathcal{T}_h; \mathbb{R}^d)$ is computed by:
\begin{equation}\label{eq:discrete-flow-heat}
    (v_h^{k+1}, w_h) + \gamma (\mathcal{I}_h(\tilde{u}_h^k \cdot v_h^{k+1}), \mathcal{I}_h(\tilde{u}_h^k \cdot w_h)) + \tau(\nabla v_h^{k+1}, \nabla w_h) = -(\nabla u_h^k, \nabla P_{h,u_h^k} w_h)
\end{equation}
for all $w_h \in \mathcal{S}_{\mathrm{D}}^1(\mathcal{T}_h; \mathbb{R}^d)$.
Similar to Section \ref{sec:harmonic-maps}, the unconstrained flow algorithm requires the following ratio to measure the stability of $P_{h,u_h^k}$:
\begin{equation}\label{eq:norm-ratio-l2}
R_{h,\gamma,L^2}[\tau, v_h^{k+1}] := 2\frac{\tau \|\nabla v_h^{k+1}\|^2+ \gamma \Vert \mathcal{I}_h(\tilde{u}_h^k \cdot v_h^{k+1})\Vert^2}{ \|\nabla P_{h,u_h^k}v_h^{k+1}\|^2}.
\end{equation}
As we will see later, $\tau \leq (1-\alpha)R_{h,\gamma,L^2}[\tau, v_h^{k+1}]$ is equivalent to the stability condition \eqref{eq:stab-criteria-heat} in Proposition \ref{prop:a-posteriori-energy-stability-heat} below.

 \begin{algorithm}[h]
\SetAlgoLined
\KwIn{initial value $u_h^0 \in \mathcal{A}_{h,0}$,
 final time $T$, penalty parameter $\gamma> 0$, adaptive time step parameter $0< \alpha<1$, initial step size $\tau_0$, and maximum step size $\tau_{\mathrm{max}}$. }

Initialize $k = 0$, $t_0=0$\;
\For{$k=0,\ldots$}{

\textbf{(1)} Compute $v_h^{k+1} \in \mathcal{S}_{\mathrm{D}}^1(\mathcal{T}_h; \mathbb{R}^d)$ to solve \eqref{eq:discrete-flow-heat}.\;
    \eIf{
    $\tau \leq (1-\alpha)R_{h,\gamma,L^2}[\tau, v_h^{k+1}]$}{
 Set $\tau_k = \tau$, $d_tu^{k+1}_h = P_{h,u_h^k}v_h^{k+1}$, $u_h^{k+1} = u_h^k + \tau_k d_t u_h^{k+1}$, and $t_{k+1} = t_k +\tau$\;

        Increase $\tau \mapsto \min\{\tau_{\mathrm{max}},  (1-\alpha)R_{h,\gamma,L^2}[\tau, v_h^{k+1}]\}$, and continue with (2)\;
    }{
        Decrease $\tau \mapsto (1-\alpha)R_{h,\gamma,L^2}[\tau, v_h^{k+1}]$, and go to (1) to repeat computation\;
    }
    \textbf{(2)}\;
    \If{$t_{k+1} \geq T$}{Stop\;}
}

\caption{Unconstrained scheme for harmonic map heat flow}
\label{alg:unconstrained-flow-heat}
\end{algorithm}

\begin{prop}[a posteriori energy stability for heat flow]\label{prop:a-posteriori-energy-stability-heat}
Let $0 < \alpha<1$, and let $\tau_k$ be a sequence of step sizes chosen such that $\tau_k\leq \tau_{\mathrm{max}}$. If the discrete scheme satisfies
\begin{equation}\label{eq:stab-criteria-heat}
     \frac{\tau_k}{2} \Vert \nabla \, P_{h, u^k_h} v_h^{k+1}\Vert^2 \leq (1- \alpha)\left[  \tau_k \Vert \nabla v_h^{k+1}\Vert^2+ \gamma \Vert \mathcal{I}_h(\tilde{u}_h^k \cdot v_h^{k+1})\Vert^2\right],
\end{equation}
then the scheme satisfies the energy estimate
\begin{equation}\label{eq:energy-estimate-heat}
    \frac{1}{2} \Vert \nabla u^{k+1}_h\Vert^2 + \tau_k\Vert v^{k+1}_h \Vert^2  +\alpha \tau_k^2\Vert \nabla v^{k+1}_h \Vert^2 + \alpha \tau_k  \gamma \Vert \mathcal{I}_h(\tilde{u}_h^k \cdot v_h^{k+1})\Vert^2 \leq \frac{1}{2}  \Vert \nabla u^{k}_h\Vert^2.
\end{equation}
If $u_h^0\in \mathcal{A}_{h,0}$, there is a constant $C>0$ such that
\begin{equation}\label{eq:constraint-violation-heat}
    \Vert \mathcal{I}_h [|u^{k+1}_h|^2 - 1]\Vert_{L^1(\Omega)} \leq C \tau_{\mathrm{max}} \Vert \nabla u^0_h\Vert^2.
\end{equation}
\end{prop}
\begin{proof}
The proof follows that of Proposition \ref{prop:a-posteriori-energy-stability} with a minor modification to account for the new stability criterion. The bound \eqref{eq:constraint-violation-heat} does not depend on $\alpha$ due to the new stability criterion.
\end{proof}

As in Section \ref{sec:harmonic-maps}, the corollary below states an a priori energy estimate for Alg.~\ref{alg:unconstrained-flow-heat}. Similar to the discussion in Remark \ref{rmk:termination}, a consequence of Corollary \ref{cor:l2-flow-stab} is that Alg.~\ref{alg:unconstrained-flow-heat} produces a sequence $\tau_k$ that is bounded from below and always satisfies \eqref{eq:stab-criteria-heat} provided $\alpha <\frac{1}{2}$.

\begin{corollary}[conditional stability of heat flow]\label{cor:l2-flow-stab}
Let $\gamma > 0$ and $0<\alpha<\frac{1}{2}$. Further assume $\tau_k = \tau$ for all $k$. There is a constant $C>0$ such that if
$$
\tau \leq C(1-2\alpha)\gamma h_{\mathrm{min}}^2,
$$
then the sequence $u^k_h$ generated by Alg.~\ref{alg:unconstrained-flow-heat} with initial condition $u_h^0\in \mathcal{A}_{h,0}$ satisfies the energy estimate \eqref{eq:energy-estimate-heat} and control over the unit length constraint violation \eqref{eq:constraint-violation-heat}.
\end{corollary}
\begin{proof}
Similar to the proof of Corollary \ref{cor:conditional-stability-h1-flow}, we must inductively show that \eqref{eq:stab-criteria-heat} is satisfied. We just show the first step as the inductive argument is a repeat of the proof of Corollary \ref{cor:conditional-stability-h1-flow}. We use \eqref{eq:proj-l2-h1-stab} from Lemma \ref{lem:proj-stab} (stability of $P_{h,u_h}$) to write the following bound
$$
\frac{\tau}{2} \Vert \nabla P_{h, u^k_h} v_h^{k+1}\Vert^2 \leq \frac{\tau}{2} \left(\Vert \nabla v^{k+1}_h\Vert +
Ch_{\mathrm{min}}^{-1}\Vert \mathcal{I}_h (v^{k+1}_h\cdot \tilde{u}_h^k)\Vert \right)^2.
$$
Using the bound $(a+b)^2 \leq (1+\epsilon)a^2 + (1+\epsilon^{-1}) b^2$, with $\epsilon = 1-2\alpha$ yields
$$
\frac{\tau}{2} \Vert \nabla P_{h, u^k_h} v_h^{k+1}\Vert^2 \leq \tau (1-\alpha)\Vert \nabla v^{k+1}_h\Vert^2 +
\tau \frac{1 - \alpha}{1-2\alpha} C^2h_{\mathrm{min}}^{-2}\Vert \mathcal{I}_h (v^{k+1}_h\cdot \tilde{u}_h^k)\Vert^2.
$$
The above bound shows that \eqref{eq:stab-criteria-heat} is satisfied provided $\tau \leq (1-2\alpha) C^{-2} h_{\mathrm{min}}^2 \gamma$.
\end{proof}

\subsection{Convergence to harmonic map heat flow}
In this section, we prove that the scheme outlined in Alg.~\ref{alg:unconstrained-flow-heat} with target dimension $d=3$ converges to a solution of harmonic map heat flow assuming quasiuniform meshes, constant step sizes $\tau_k = \tau$, and $\tau \leq C(1-2\alpha)\gamma h^2$. For $0< \alpha < \frac{1}{2}$, the proofs and statements can also be adapted to $d\neq3$ and nonconstant step sizes with minor modifications. Except for Lemma \ref{lem:consistency-error}, the arguments in this section are standard and closely follow those found in the proof of \cite[Theorem 7.7]{bartels2015numerical}. Stating the main theorem requires defining a few interpolants in time.

\begin{definition}[interpolants in time]\label{def:time-interp}
Given a sequence $\{u^k\}_{k=0}^\infty\subset  L^2(\Omega)$, we define its piecewise constant and piecewise linear interpolants in time for $t\in (\tau k , \tau (k+1))$ by
$$
\bar{u}^-_{\tau}(t) = u^k, \quad \bar{u}^+_{\tau}(t) = u^{k+1}, \quad \hat{u}_{\tau}(t) = \frac{\tau (k+1) - t }{\tau} u^k + \frac{t - \tau k}{\tau} u^{k+1}.
$$
\end{definition}

\begin{theorem}[convergence to harmonic map heat flow]\label{thm:convergence-heat-flow}
Fix $d=3$. Assume $\mathcal{T}_h$ is a quasiuniform sequence of meshes, $0<\alpha<\frac{1}{2}$, and $\tau\leq C(1-2\alpha)\gamma h^2$ to satisfy the hypothesis of Corollary \ref{cor:l2-flow-stab}. Assume $\gamma\geq c>0$, and $u_h^0$ uniformly bounded in $H^1(\Omega)$, with $u^0_h\to u_0$ strongly in $L^2(\Omega)$. Define $\bar{u}^-_{h,\tau}$ as the piecewise constant interpolant of the solution $u^k_h$ in time. The sequence $\{\bar{u}^-_{h,\tau}\}_{h, \tau}$ is uniformly bounded in $L^\infty(0,T;H^1(\Omega))$, and any weak-$\star$ $L^\infty(0,T;H^1(\Omega))$ accumulation point $u$ of $\{\bar{u}^-_{h,\tau}\}_{h, \tau}$ satisfies $u\in L^\infty((0,T); \mathcal{A})\cap H^1((0,T); L^2(\Omega))$. Any such point $u$ is a weak solution to harmonic map heat flow in the sense that $u = u_0$ at $t=0$, and
\begin{equation}\label{eq:heat-flow-def}
(\dot{u}, w) + (\nabla u, \nabla w) = 0 \text{ a.e.\ in } (0,T),
\end{equation}
for all $w\in L^\infty((0,T); H^1_{\mathrm{D}}(\Omega))$ that satisfy $w\cdot u = 0$ a.e.\ in $(0,T)\times \Omega$.
\end{theorem}
\begin{proof}
We should first note that $u\in L^\infty((0,T); \mathcal{A})\cap H^1((0,T); L^2(\Omega))$ is a solution to harmonic map heat flow if and only if
$$
	(\partial_t u , u\times \phi ) + (\nabla u,\nabla  (u\times \phi)) = 0 \text{ for a.e. } t\in (0,T)
$$
for all $\phi \in L^\infty((0,T); H^1_{\mathrm{D}}(\Omega)\cap L^\infty(\Omega))$ \cite[Proposition 7.5]{bartels2015numerical}. 

As a consequence of the a priori bounds in Corollary \ref{cor:l2-flow-stab}, the sequence $\{\bar{u}^-_{h,\tau}\}_{h, \tau}$ is uniformly bounded in $L^\infty(0,T;H^1(\Omega))$ as $h,\tau\to0$. To prove that weak-$\star$ accumulation points of $\{\bar{u}^-_{h,\tau}\}_{h, \tau}$ are elements of $L^\infty((0,T); \mathcal{A})\cap H^1((0,T); L^2(\Omega))$ and satisfy the initial condition, we refer to the standard compactness arguments found in Steps 1--4 of the proof of \cite[Theorem 7.7]{bartels2015numerical}. Finally, the result follows from harmonic map heat flow compactness arguments in the proof of \cite[Theorem 7.7]{bartels2015numerical}, consistency in Lemma \ref{lem:consistency-error}, and a density argument.
\end{proof}

The main difference between the convergence of the unconstrained scheme and the classical projection-free scheme is an extra consistency error, handled in the lemma below.

\begin{lemma}[consistency error] \label{lem:consistency-error}
Let $\phi\in C^\infty([0,T]\times \overline{\Omega})$ and let $\bar{\phi}^-_\tau$ be its piecewise constant interpolation in time. Consider the test function $w_{h,\tau} = \mathcal{I}_h [\bar{u}^-_{h,\tau} \times \bar{\phi}^-_\tau]$. Denote $a_{h,\tau} = \mathcal{I}_h[\bar{\tilde{u}}^-_{h,\tau} \bar{\tilde{u}}^-_{h,\tau} \cdot \bar{v}^+_{h,\tau}]$. For $T = K\tau$, the following equality holds:
\begin{equation}\label{eq:integral-test-phi}
\int_0^T  (\partial_t\hat{u}_{h,\tau}, {w_{h,\tau}})  + (\nabla \bar{u}^-_{h,\tau} ,\nabla {w_{h,\tau}}) dt = - \tau  \int_0^T   (\nabla \bar{v}^+_{h,\tau}, \nabla {w_{h,\tau}})dt
-  \int_0^T (a_{h,\tau}, {w_{h,\tau}}) dt.
\end{equation}
Additionally, if $\tau \leq C(1-2\alpha) \gamma h^2$ to satisfy the hypothesis of Corollary \ref{cor:l2-flow-stab}, and $\gamma\geq c >0$, then the RHS converges to zero as $\tau, h \to0$.
\end{lemma}
\begin{proof}
The proof requires 3 steps.

\noindent {\it Step 1. Proof of \eqref{eq:integral-test-phi}}: Test equation \eqref{eq:discrete-flow-heat} with $w^k_h = \mathcal{I}_h [u^k_h \times {\phi}^k]$. After recognizing $P_{h,u^k_h}v^{k+1}_h =  d_t u ^{k+1}_h$ and $\mathcal{I}_h[u^k_h\cdot w^k_h] = 0$, the following holds
\begin{align*}
(d_t u ^{k+1}_h, \mathcal{I}_h [u^k_h \times \phi^{k}]) &=(v^{k+1}_h, \mathcal{I}_h [u^k_h \times \phi^{k}]) - (\mathcal{I}_h[\tilde{u}^k_h \tilde{u}^k_h\cdot v^{k+1}_h], \mathcal{I}_h [u^k_h \times \phi^{k}])\\
&= - (\nabla (u^k_h +\tau v^{k+1}_h),\nabla \mathcal{I}_h [u^k_h \times \phi^{k}])- (\mathcal{I}_h[\tilde{u}^k_h \tilde{u}^k_h\cdot v^{k+1}_h], \mathcal{I}_h [u^k_h \times \phi^{k}])
\end{align*}
Hence, multiplying by $\tau$, summing from $k=0,\ldots, K-1$, and rewriting as an integral from $0$ to $T$ leads to \eqref{eq:integral-test-phi}.

\noindent {\it Step 2. First term on RHS of \eqref{eq:integral-test-phi} vanishes}: The proof that the first term vanishes as $\tau,h\to0$ follows standard arguments for harmonic map heat flow and can be found in Step 5 of the proof of \cite[Theorem 7.7]{bartels2015numerical}. We will summarize for completeness. The ingredients of the proof include showing $w_{h,\tau}$ is uniformly bounded in $L^2((0,T);H^1(\Omega))$ as $\tau,h\to0$ and using the numerical dissipation in \eqref{eq:energy-estimate-heat} to show that $\Vert \tau^{1/2} \bar{v}^+_{h,\tau}\Vert_{L^2((0,T);H^1(\Omega))}$ is uniformly bounded. Putting it all together, $\int_0^T  \tau (\nabla \bar{v}^+_{h,\tau}, \nabla w_{h,\tau} )dt$ is $O(\sqrt{\tau})$, and hence the first term vanishes in the limit.

\noindent {\it Step 3. Second term on RHS \eqref{eq:integral-test-phi} vanishes}: For notational brevity, let $X := L^2((0,T); L^2(\Omega))$. Recall that $\bar{\tilde{u}}^-_{h,\tau} \cdot (\bar{u}^-_{h,\tau} \times \bar{\phi}^-_\tau)=0$ on nodes, which implies $\mathcal{I}_h [a_{h,\tau}\cdot w_{h,\tau}] = 0$. Hence, the second term of the RHS of \eqref{eq:integral-test-phi} is bounded by:
\begin{align*}
\bigg| \int_0^T (a_{h,\tau}, w_{h,\tau}) dt\bigg| &\leq \Vert a_{h,\tau}\cdot w_{h,\tau} - \mathcal{I}_h [a_{h,\tau}\cdot w_{h,\tau}] \Vert_{L^1((0,T); L^1(\Omega))} \leq Ch^2 \Vert \nabla a_{h,\tau}\Vert_{X} \, \Vert \nabla w_{h,\tau}\Vert_{X},
\end{align*}
where the last interpolation estimate is proven using arguments similar to \cite[Lemma 2.1]{bartels2016projection} and Lemma \ref{lem:proj-stab}.
Using the uniform boundedness of $\Vert \nabla w_{h,\tau}\Vert_{X}$, the identity $a_{h,\tau} = \mathcal{I}_h[\bar{\tilde{u}}^-_{h,\tau} \bar{\tilde{u}}^-_{h,\tau} \cdot \bar{v}^+_{h,\tau}]$, an inverse estimate, and a nodal bound $|\bar{\tilde{u}}^-_{h,\tau}(z)| \leq 1$ leads to
$$
\Vert \nabla a_{h,\tau}\Vert_{X} \leq \frac{C}{h_{\mathrm{min}}\sqrt{\gamma}}  \sqrt{\gamma} \Vert \mathcal{I}_h ( \bar{\tilde{u}}^-_{h,\tau} \cdot \bar{v}^+_{h,\tau})\Vert_{X}.
$$
The energy estimate \eqref{eq:energy-estimate-heat} implies $\sqrt{\gamma} \Vert \mathcal{I}_h ( \bar{\tilde{u}}^-_{h,\tau} \cdot \bar{v}^+_{h,\tau})\Vert_{X}$ is uniformly bounded. If $\gamma\geq c>0$ and $\mathcal{T}_h$ is quasiuniform, then $h^2 \Vert \nabla a_{h,\tau}\Vert_{X} \lesssim c^{-1/2} h \to0$ as $h\to 0$, which completes the proof. 
\end{proof}

\section{Computations}\label{sec:computations}

In this section, we explore computational performance of the new scheme compared with the projection-free scheme outlined in Section \ref{sec:intro-proj-free}. For the projection-free scheme, we solve the typical discrete version of \eqref{eq:saddle-point-problem} for $(d_tu^k_h, \lambda^k_h)\in T\mathcal{A}_{h,\mathrm{D}}(u_h^k)\times \mathcal{S}_\mathrm{D}^1(\mathcal{T}_h;\mathbb{R})$ using MinRes for the saddle point system, combined with an augmented Lagrangian preconditioner similar to \cite{xia2021augmented}. We stop the iteration when $\Vert d_tu^k_h\Vert_* < \varepsilon$ or when $t_k \geq T$, just like Alg.~\ref{alg:unconstrained-flow} and Alg.~\ref{alg:unconstrained-flow-heat}. Both schemes were implemented using NGSolve \cite{schoberl2014c++}, and the computations in this section were run on a computer with an Apple M2 Pro CPU with 32GB RAM. While comparing runtimes can be sensitive to implementation details, we believe the computational evidence for the improved efficiency in this section is conclusive, and point to the dissertation of the third author \cite{palus2024finite} for more computational evidence of the improved efficiency of this scheme.

\subsection{Example 1: smooth harmonic map heat flow}\label{sec:conv-study}

This example is a convergence study of Alg.~\ref{alg:unconstrained-flow-heat}, following the example of \cite[Section 9.2]{akrivis2021higher}. We consider the time interval $[0,T] = [0, .2]$ and spatial domain $\Omega = (0,1)^2$. After defining $\beta(t) = (T +0.1)/(T +0.1 - t)$ and $d(x) = (x_1 - 1/2)^2+ (x_2-1/2)^2$, the manufactured solution $u:[0,T]\times \Omega \to \mathbb{R}^3$ with $A = 100$ is
$$
    u(t,x) =
    \begin{cases}
        \frac{A}{2} \mathrm{e}^{-\frac{\beta(t)}{1 / 4-d(x)}}\tilde{\nabla d(x)} \;
        						+\;\big(1-A^{2} \mathrm{e}^{-2 \frac{\beta(t)}{1 / 4-d(x)}} d(x)\big)^{1/2} \; e_3, & d(x) \leq \frac{1}{4} \\
        e_3, & \text{ otherwise }
    \end{cases},
$$
where $\tilde{\nabla d(x)} = (\nabla d(x)^\top, 0)^\top$.
Clearly, $|u(t,x)| = 1$ for all $(t,x)$. The Dirichlet boundary data is $g = e_3$ on $\partial\Omega$ and the RHS forcing is
$
	f := u_t - \Delta u - |\nabla u|^2u.
$
We denote the discrete solution as $u_h$, which is defined at time steps $t_j = \tau j$. Fig.~\ref{fig:convergence-study} plots the discrete errors
\begin{align*}
    \Vert u_h - u\Vert^2_{L^2_\tau(0,T;H^1(\Omega))} &= \tau\sum_{j=0}^{T/\tau} \Vert u_h(t_j) - u(t_j)\Vert_{H^1(\Omega)}^2, \\
     \Vert u_h - u\Vert_{L^\infty_\tau(0,T;L^2(\Omega))} &= \max_{j=0, \ldots, T/\tau} \Vert u_h(t_j) - u(t_j)\Vert_{L^2(\Omega)}
\end{align*}
for $h = 2^{-k}$ with $k=2,\ldots, 8$, $\tau=\frac{4}{5} h, \frac{16}{5} h^2$, and $\gamma = h^{0},h^{-1}, h^{-2}$. In Fig.~\ref{fig:convergence-study}, we see that the errors empirically satisfy $\Vert u_h - u\Vert_{L^2_\tau(0,T;H^1(\Omega))} \leq C(h + \tau)$ and $\Vert u_h - u\Vert_{L^\infty_\tau(0,T;L^2(\Omega))} \leq C(h^2 + \tau)$, which is optimal for piecewise affine elements and implicit time-stepping. 
\begin{figure}[h]
    \includegraphics[width=.85\textwidth]{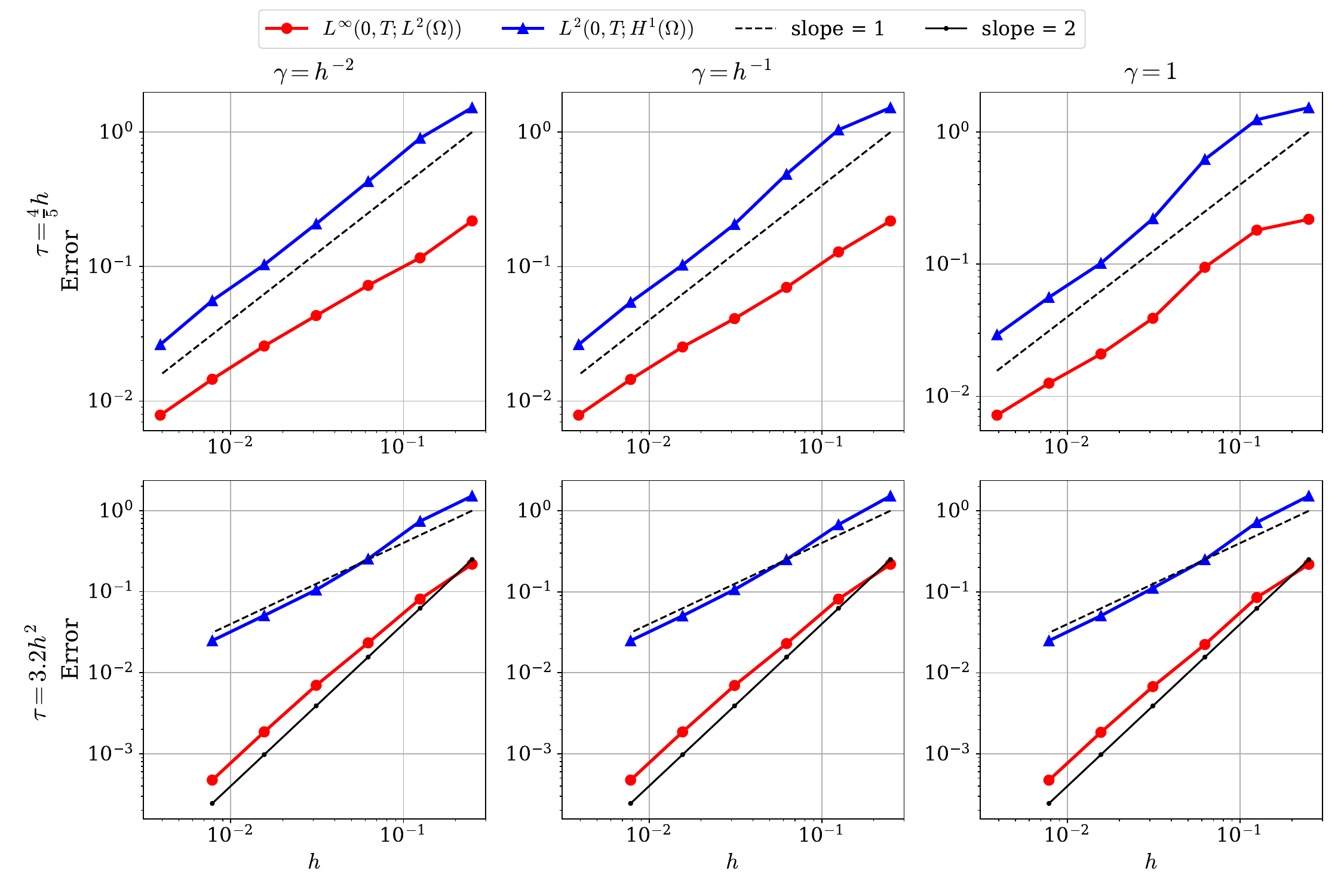}
    \caption{Example 1 (smooth harmonic map heat flow): Error between the discrete solution $u_h(t)$ and the exact solution $u(t)$ in $L^2_\tau(0,T;H^1(\Omega))$ and $L^\infty_\tau(0,T; L^2(\Omega))$ norms vs $h$ for different relationships of $\gamma$ and $\tau$. First row is $\tau = \frac{4}{5} h$ and second row is $\tau = \frac{16}{5} h^2$. The columns (left to right) correspond to $\gamma = h^{-2}, h^{-1}, 1$ respectively. The reference slopes of 1 and 2 show empirically that $\Vert u_h - u\Vert_{L^2_\tau(0,T;H^1(\Omega))} \leq C(h + \tau)$ and $\Vert u_h - u\Vert_{L^\infty_\tau(0,T;L^2(\Omega))} \leq C(h^2 + \tau)$, which is optimal for piecewise affine elements and implicit time-stepping.}
    \label{fig:convergence-study}
\end{figure}

\subsection{Example 2: singular heat flow}

Following the example of \cite[Example 5.1]{barrett2007convergent}, let $\Omega = (-1,1)^2$. We seek a harmonic map heat flow $u:[0, .5] \times \Omega \to \mathbb{R}^3$ with initial condition written in polar coordinates $(r, \varphi)$ as
$$
u_0(x) = \big(\cos(\varphi)\sin(\phi(r)),\; \sin(\varphi)\sin(\phi(r)),\; \cos(\phi(r))\big)^\top, \quad \phi(r) = \frac{3\pi}{2} r^2.
$$
This problem experiences a finite time blowup in the sense that there is a time $t_s<\infty$ where $\Vert \nabla u(t_s,\cdot)\Vert_{L^{\infty}(\Omega)} = \infty$ \cite{chang1992finite}. Due to the singularity, the inf-sup constant of the saddle point system \eqref{eq:saddle-point-problem} of the projection-free scheme will degenerate near $t_s$. Hence, a scheme that only requires SPD solves is particularly attractive for this example.

In these experiments, the mesh has $h_{\mathrm{max}} = 1/16$ near the outer edges of the square and $h_{\mathrm{min}} = 1/64$ in a disk of radius $1/4$ centered at the origin to resolve the singularity. This mesh is shown in Fig.~\ref{fig:blowup}. We then test multiple values of $\tau_{\mathrm{max}}$ and $\alpha=0.5,0.9$ for both  Alg.~\ref{alg:unconstrained-flow-heat} with $\gamma = h_{\mathrm{min}}^{-1}$ and the projection-free flow outlined in Section \ref{sec:intro-proj-free}. Although $\alpha = 0.5,0.9$ do not satisfy the hypothesis required for Corollary~\ref{cor:l2-flow-stab} and Theorem~\ref{thm:convergence-heat-flow}, the stability criterion \eqref{eq:stab-criteria-heat} is still enforced at every step, so the computed iterates still satisfy the energy estimate and constraint-violation bound of Proposition \ref{prop:a-posteriori-energy-stability-heat}. Only the a priori lower bound on $\tau_k$ is not guaranteed. These computations will show that aggressive step-size decreases appear valuable near singularities.

Fig.~\ref{fig:blow-up-defects} shows the discrete solutions to the different schemes at their blowup times $t_s$ on the subdomain $(-0.1,0.1)^2$ for $\tau, \tau_{\mathrm{max}} = 0.007812$. The continuous solution in this example is symmetric about the origin. In the case of $\alpha = 0.5, 0.9$, the discrete solutions are nearly symmetric and are qualitatively similar to the projection-free solution. However, the solution computed by Alg.~\ref{alg:unconstrained-flow-heat} with constant time steps lacks the same symmetry. These results indicate that using the variable time-stepping scheme helps to preserve properties of the harmonic map heat flow solution.

Table \ref{tab:blowup-table0} shows the reported data. We see that the unconstrained flow with constant step sizes performs very similarly to the classical projection-free scheme but is approximately 6 to 7 times more efficient in terms of computation time. As $\alpha$ increases to $0.5$ and $0.9$, we see an increase in the number of steps, $N$, with the additional benefit of a much smaller unit length constraint violation in $L^\infty$ and $L^1$ when $\alpha=0.9$. Additionally, Alg.~\ref{alg:unconstrained-flow-heat} with $\alpha=0.9$ is at most $2\times$ slower than the projection-free scheme, but the new scheme has $5\times$ to $7\times$ smaller unit length constraint violation errors. Compared to both the projection-free scheme and Alg.~\ref{alg:unconstrained-flow-heat} with constant step sizes, Alg.~\ref{alg:unconstrained-flow-heat} with $\alpha=0.9$ also has a less $\tau$-dependent singular time $t_s$. In Fig.~\ref{fig:blowup-charts}, the $W^{1,\infty}$ seminorm of the solution starts to increase and the upper bounds in Lemma \ref{lem:proj-stab} suggest that the projection $P_{h,u_h}$ potentially loses stability. As a result, Alg.~\ref{alg:unconstrained-flow-heat} reduces $\tau_k$.

\begin{figure}
\includegraphics[height=4cm]{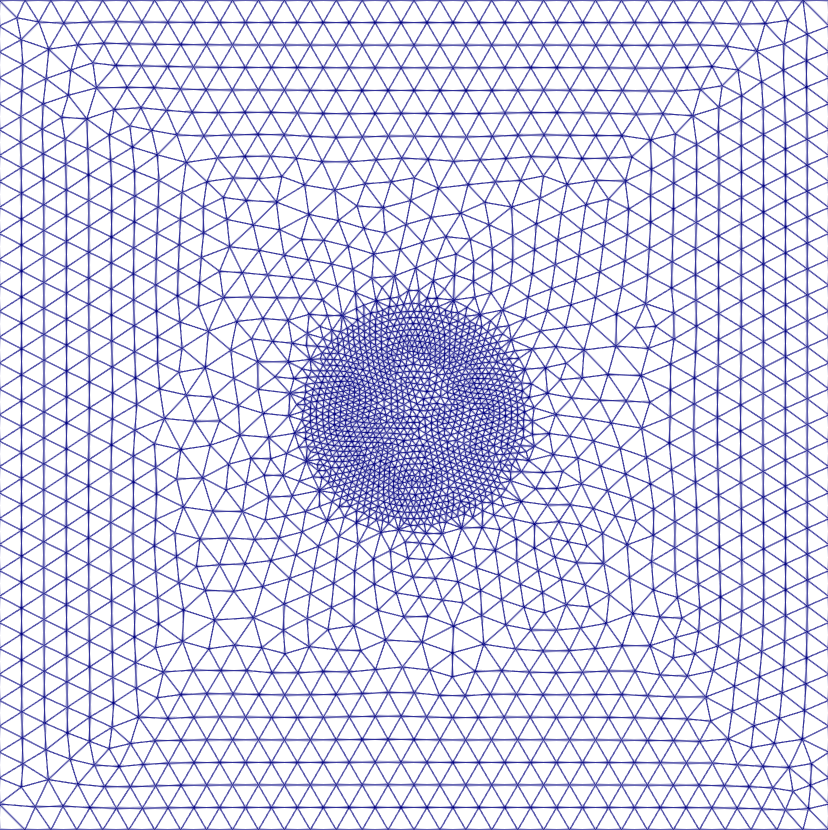}
\includegraphics[height=4cm]{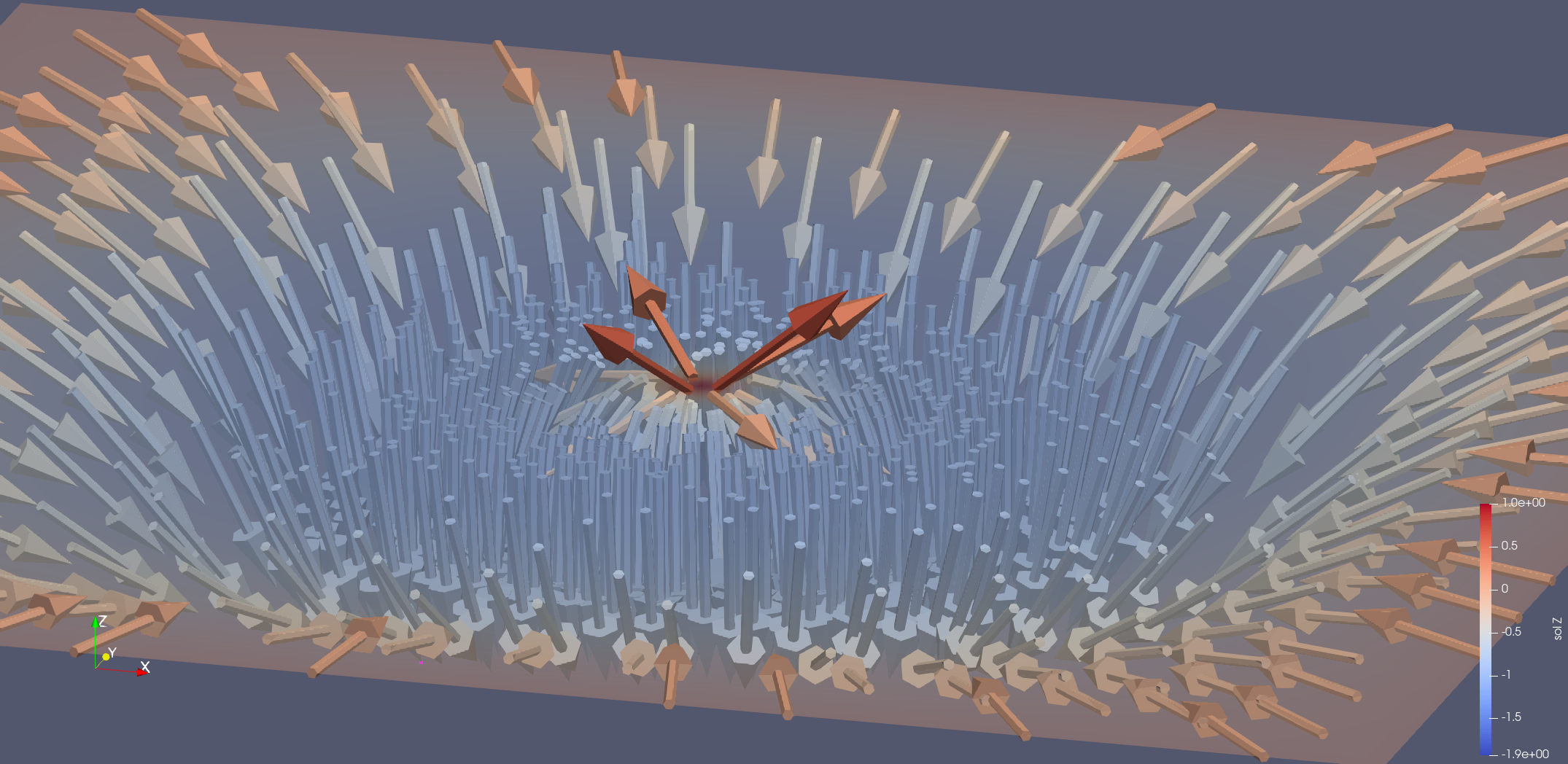}
\caption{Example 2 (singular heat flow). (Left): Mesh of $\Omega$.  (Right): Director field immediately before singularity formation.}
\label{fig:blowup}
\end{figure}
\begin{figure}
\includegraphics[height=6cm]{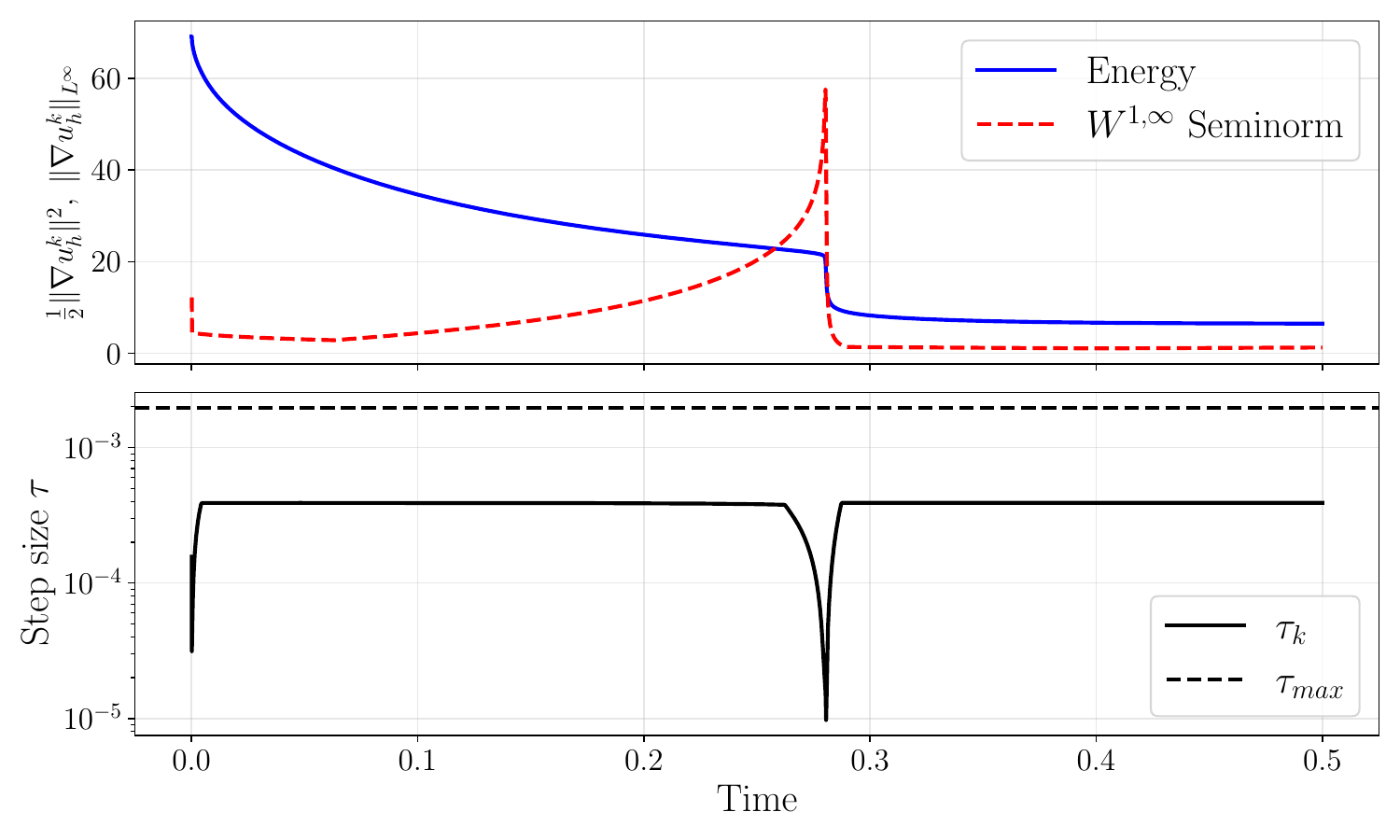}
\caption{Example 2 (singular heat flow). Top to bottom: Energy and $W^{1,\infty}$ seminorm of the solution vs time for the unconstrained flow with $\alpha = .9$ and $\tau_{\mathrm{max}} = 0.001953$. Semilog plot of stepsize $\tau_k$ vs time. We see that as the $W^{1,\infty}$ seminorm of the solution starts to increase, Alg.~\ref{alg:unconstrained-flow-heat} detects the loss of stability of the projection $P_{h,u_h}$, which matches the upper bound seen in Lemma \ref{lem:proj-stab}. Additionally, when $\tau_k$ is larger, we see that $\tau_k \approx \frac{\tau_{\mathrm{max}}}{10}$. This is because the predicted velocity $v^{k+1}_h$ computed with step size $\tau_{\mathrm{max}}$ does not satisfy the stability condition \eqref{eq:stab-criteria-heat}, and the algorithm decreases $\tau_k$ to $\approx (1-\alpha)\tau_{\mathrm{max}} = \frac{\tau_{\mathrm{max}}}{10}$.}
\label{fig:blowup-charts}
\end{figure}

\begin{figure}
\includegraphics[width=\textwidth]{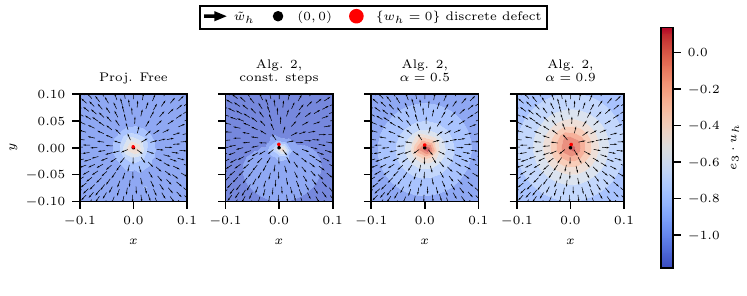}
\caption{Example 2 (singular heat flow). Discrete solutions at blowup time $t_s$ on subdomain $(-0.1,0.1)^2$ for $\tau, \tau_{\mathrm{max}} = 0.007812$. The black dots display the origin, and the red dots show the set $\{ w_h = 0\}$ of the projected vector field $w_h = u_h - (u_h\cdot e_3) e_3$, which serves as a discrete topological defect that should be located at the origin. The color map shows $e_3\cdot u_h$, and the vector field displays the normalized field $\tilde{w}_h$. Left to Right: The projection-free solution, Alg.~\ref{alg:unconstrained-flow-heat} solution with constant time steps and variable time steps with $\alpha = 0.5, 0.9$.}
\label{fig:blow-up-defects}
\end{figure}

\begin{table}
    \centering
    \caption{Example 2 (singular heat flow). Table of $\tau$, discrete blowup time $t_s$, unit length constraint violations, number of iterations, final energy, and computation time for different schemes. Using Alg.~\ref{alg:unconstrained-flow-heat} with constant step sizes leads to an approximately $6\times$--$7\times$ speed-up. With variable step sizes, $\alpha=0.5,0.9$, there is a trade-off between number of iterations and the unit length constraint violation.}
    \label{tab:blowup-table0}
    \input{data/blowup_level0_combined.tex}
\end{table}

\subsection{Example 3: liquid crystal shell}

This example considers a 3D domain $\Omega = B_1((0,0,.15))\setminus B_{1/2}((0,0,0))$, which represents the body of a shell of liquid crystal. In experiments \cite{fernandez2007novel, lopez2011frustrated}, the liquid crystal is immersed in a fluid like water and a droplet of water can be immersed inside the shell. At the interface of water and liquid crystal ($\partial\Omega$ in this case), one often expects tangential anchoring of the liquid crystal, which forces topological defects to form at the interface. The total topological charge of these defects will add up to 2 on each sphere for topological reasons. There are four $+1/2$ defects on each sphere in the case of thin shells \cite{fernandez2007novel}, which makes a director field model inappropriate. However, in the case of thick shells, experiments show that there can be two $+1$ defects on each sphere \cite{lopez2011frustrated}. Director field models can describe this precise situation.

The total energy of the liquid crystal with a weak tangential anchoring surface energy is
$$
E[u] = \frac{1}{2} \int_\Omega |\nabla u|^2 dx + c_{\mathrm{pen}} E_{\mathrm{anch}}[u], \quad E_{\mathrm{anch}}[u] = \frac{1}{2} \int_{\partial\Omega} (u\cdot \nu)^2  d S.
$$
Since $E_{\mathrm{anch}}$ is quadratic, treating this term implicitly leads to a stable projection-free scheme as well as a simple modification of Alg.~\ref{alg:unconstrained-flow}. In this example, we set $c_{\mathrm{pen}} = 100$, and we set $\alpha = 0.9, \gamma=0,$ and $\tau_{\mathrm{max}} = 1$ as the numerical parameters for Alg.~\ref{alg:unconstrained-flow}. For both Alg.~\ref{alg:unconstrained-flow} and the projection-free scheme, we consider the stopping tolerance $\varepsilon = 5\times 10^{-3}$. Also, we use the full $H^1$ norm for the flow metric $(u,v)_* = (u,v)_{H^1(\Omega)}$ due to the lack of Dirichlet boundary conditions. For the projection-free scheme and Alg.~\ref{alg:unconstrained-flow} with constant $\tau$, we computed a time step $\tau$ using a root finding algorithm so that the $L^1$ unit length constraint violation matches that of Alg.~\ref{alg:unconstrained-flow} with $\alpha=0.9$ on a coarser mesh. Despite setting $\gamma=0$, we did not decouple the vector component solves in \eqref{eq:discrete-flow} in Alg.~\ref{alg:unconstrained-flow}. However, the iterative solution of \eqref{eq:discrete-flow} in the case $\gamma=0$ is easily parallelized.

Fig.~\ref{fig:lc-shell} (left) shows a slice of the domain $\Omega$ with the two confining spheres as well as the solution $u_h^\infty$ on the spheres. We can see that +1 defects form on each sphere. Fig.~\ref{fig:lc-shell} (right) shows the step size history. Although $\tau_{\mathrm{max}} = 1$, the algorithm never increases $\tau_k$ to $1$ in order to maintain stability.

Table \ref{tab:lc-shell} reports the minimum and maximum step size, final unit length constraint error, final energy, number of iterations, and computation time for the unconstrained scheme and the projection-free scheme. Alg.~\ref{alg:unconstrained-flow} with constant step sizes provides a $3.9\times$-$7.6\times$ speed-up in computation time over the projection-free scheme for comparable $\Vert \mathcal{I}_h |u_h^\infty|^2 - 1\Vert_{L^1}$ and $E[u^\infty_h]$. The speed-up factor improves with mesh refinement, which comes from solving SPD systems in the unconstrained scheme compared with solving saddle point systems in the classical projection-free scheme. The projection-free scheme and Alg.~\ref{alg:unconstrained-flow} with constant step sizes take a similar number of iterations, but the cost per iteration of Alg.~\ref{alg:unconstrained-flow} is about $6\times$ smaller for $\#\mathrm{DOF}=32178$.

Also for $\#\mathrm{DOF}=32178$, the unconstrained scheme with adaptive time-stepping and $\alpha = 0.9$ reduced the computation time by 48$\times$ over the projection-free scheme and provided a $6\times$ cost reduction over Alg.\ \ref{alg:unconstrained-flow} with constant step sizes. The cost per iteration for $\alpha=0.9$ is about $2\times$ larger than for constant step sizes. This is because the time step increase of Alg.\ \ref{alg:unconstrained-flow} is often too aggressive, which requires recomputing and decreasing $\tau$ in the next step. This extra computation effectively doubles the cost per iteration over constant step sizes. Despite this extra cost, Alg.\ \ref{alg:unconstrained-flow} with $\alpha = 0.9$ can substantially increase the step size when stability is less binding, and the optimization takes many fewer iterations. In the case of this LC shell example, $\tau_k\approx 0.1$ for large $k$, cf.\ Fig.~\ref{fig:lc-shell}.

\begin{figure}
\includegraphics[height=4cm]{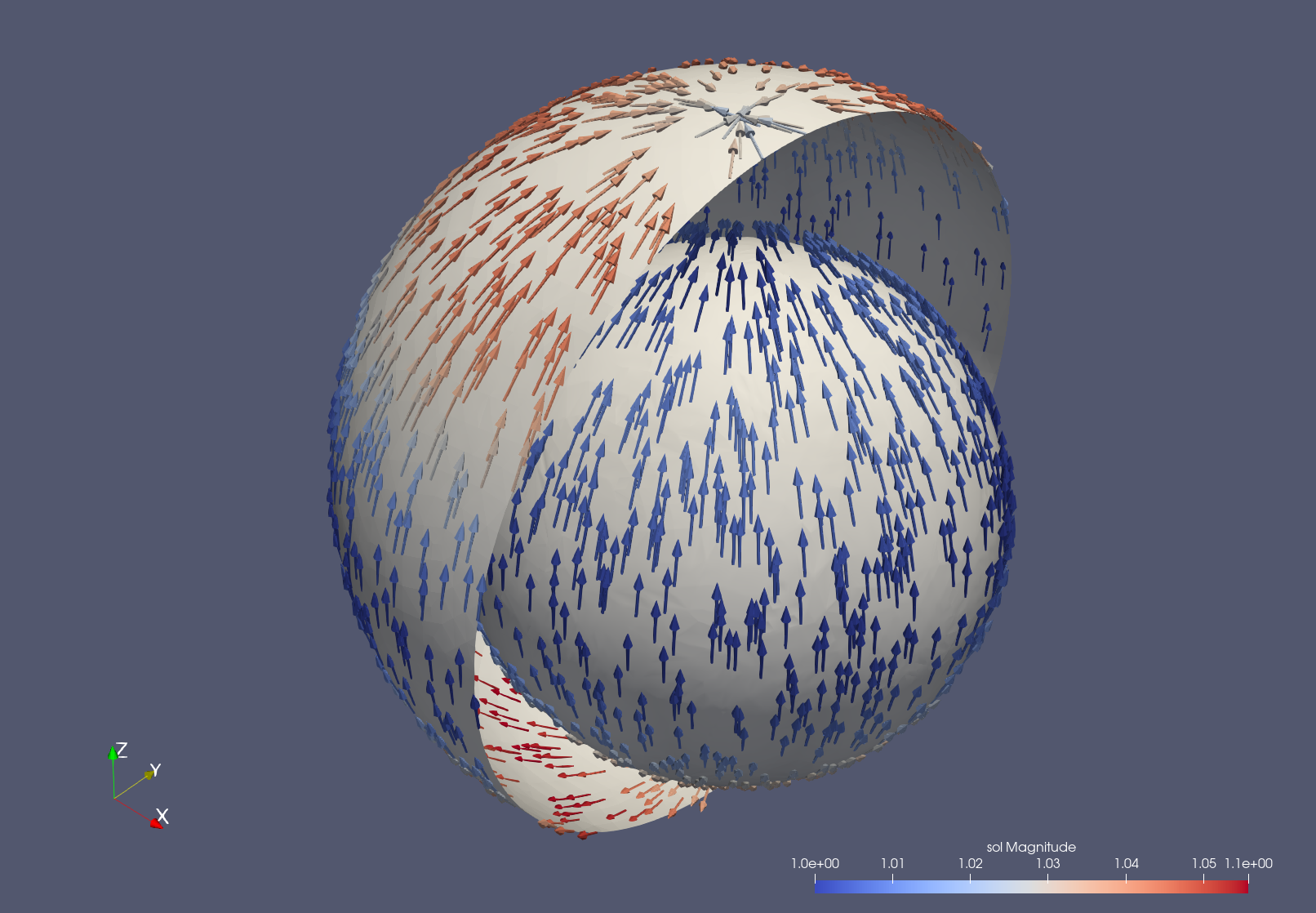} \includegraphics[height=4cm]{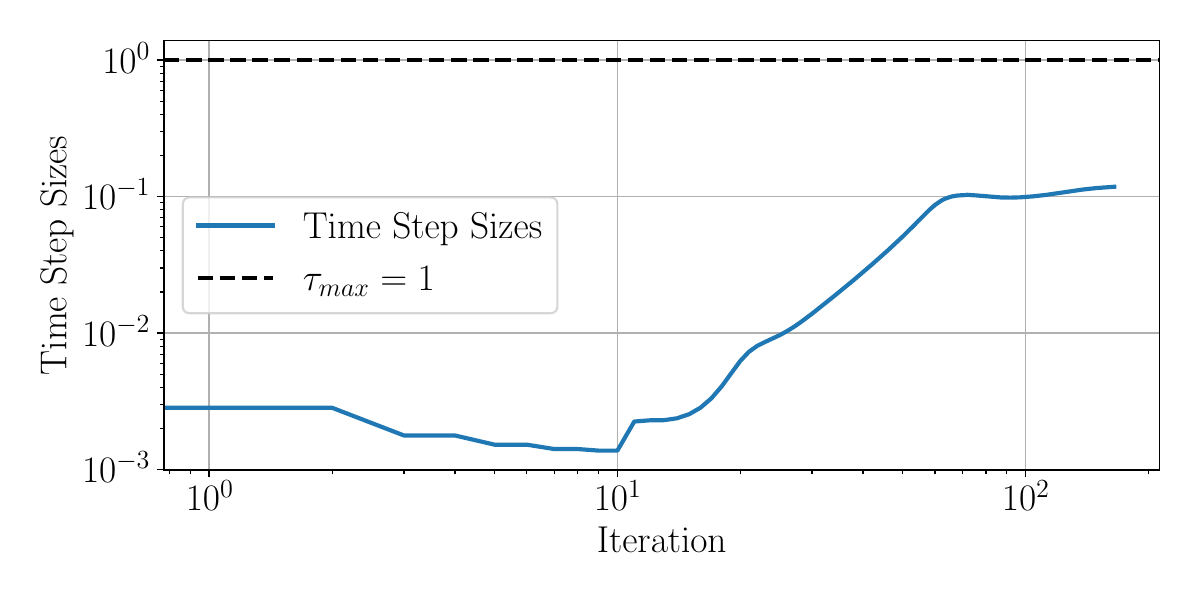}
\caption{Example 3 (LC shell). Left: slice of the domain $\Omega$ with the two confining spheres and the solution $u_h^\infty$ on the spheres produced by Alg.~\ref{alg:unconstrained-flow} on the finest mesh. The solution has two $+1$ defects on each sphere for a total of four defects. Right: Log-log plot of step size vs iteration of Alg.~\ref{alg:unconstrained-flow} with $\tau_{\mathrm{max}}=1, \alpha=0.9$ on the finest mesh.}
\label{fig:lc-shell}
\end{figure}

\begin{table}
\centering
\caption{Example 3 (LC shell). Number of degrees of freedom $\#\mathrm{DOF} = \#\mathrm{DOF}(\mathcal{S}^1(\mathcal{T}_h;\mathbb{R}^3))$,  minimum and maximum step size, final unit length constraint error, final energy, number of iterations, $N$, computation time, and average computation time per iteration for the LC shell experiment. The last row for the projection-free scheme is omitted due to compute time.}
\label{tab:lc-shell}

\begin{tabular}{cccccccc}
\toprule
 $\#\mathrm{DOF}$ & $\min \tau_k$ & $\max \tau_k$ & $\Vert \mathcal{I}_h |u_h^\infty|^2 - 1 \Vert_{L^1}$ & $N$ & $E[u^\infty_h]$ & Time (s) & Time/iteration (s) \\
\hline
\multicolumn{8}{c}{projection-free scheme with $\tau = 0.00837$ } \\
\hline
798 & 0.00837 & -- & 0.137 & 10523 & 14.1 & 508 & 0.0483\\
4776 & 0.00837 & -- & 0.123 & 4814 & 12.9 & 1378 & 0.286\\
32178 & 0.00837 & -- & 0.120 & 2381 & 10.8 & 8458 & 3.55\\
233982 	& -- & -- & -- & -- & -- & -- & --\\
\hline
\multicolumn{8}{c}{Alg.~\ref{alg:unconstrained-flow} with constant step size $\tau =0.00680$ } \\
\hline
798 		& 0.00680 & -- & 0.135 & 8376 & 14.1 & 112 & 0.0134\\
4776 	& 0.00680 & -- & 0.123 & 4170 & 12.9 & 354 & 0.0849\\
32178 	& 0.00680 & -- & 0.120 & 1894 & 10.9 & 1110 & 0.586\\
233982 	& 0.00680 & -- & 0.120 & 1711 & 10.3 & 9310 & 5.44\\
\hline
\multicolumn{8}{c}{Alg.~\ref{alg:unconstrained-flow} with $\alpha = 0.9, \tau_{\mathrm{max}} = 1$} \\
\hline
798 		& 0.00141 & 0.122 & 0.142 & 545 & 14.2 & 14 & 0.0257\\
4776 	& 0.00139 & 0.136 & 0.120 & 288 & 12.9 & 41 & 0.142\\
32178 	& 0.00138 & 0.155 & 0.114 & 162 & 10.8 & 176 & 1.086\\
233982 	& 0.00138 & 0.118 & 0.112 & 168 & 10.3 & 1865 &  11.1 \\
\bottomrule
\end{tabular}
\end{table}
\section{Computational extension to plate bending}\label{sec:plate-bending}
In this section, we computationally demonstrate the effectiveness of Alg.~\ref{alg:unconstrained-flow} when applied to a nonlinear plate bending problem. An analysis of the scheme in this setting is left to future work. The computations presented in the following were run on an Intel\textsuperscript{\textregistered} notebook featuring a Core\textsuperscript{\texttrademark} i5-7300U 2.6\,GHz CPU with 16\,GB RAM.

\subsection{Isometric plate bending}
Let $\Omega \subset \mathbb{R}^2$ be a bounded domain describing a plate in its flat reference configuration. On a subset $\Gamma_D \subset \partial\Omega$, clamped boundary conditions are imposed: any deformation $y$ and its gradient are required to attain prescribed values $y_D$ and $\Phi_D$ on $\Gamma_D$. Writing
\[
H^2_{y_D,\Phi_D}(\Omega;\mathbb{R}^3) = \{ y \in H^2(\Omega;\mathbb{R}^3) : y = y_D \text{ and } \nabla y = \Phi_D \text{ on } \Gamma_D\}
\]
for the space incorporating the boundary conditions, and denoting by $H^2_D(\Omega;\mathbb{R}^3)$ the corresponding homogeneous space, the admissible set of isometric deformations is
\begin{equation*}
\mathcal{A}_{y_D, \Phi_D} = \{ y \in H^2_{y_D,\Phi_D}(\Omega;\mathbb{R}^3) : \nabla y^\top \nabla y = I_{2\times 2} \text{ a.e.\ in } \Omega\}.
\end{equation*}
Given a body force $f : \Omega \to \mathbb{R}^3$, the equilibrium deformation minimizes the bending energy
\begin{equation*}
E_{\mathrm{bend}}[y] = \int_\Omega \frac{1}{2}|\nabla \nabla y|^2 - f\cdot y \; dx \quad \text{subject to } y \in \mathcal{A}_{y_D, \Phi_D}.
\end{equation*}
To discretize the fourth order problem, we employ Discrete Kirchhoff Triangle (DKT) elements~\cite{braess2007finite}, whose finite element space is a subspace of the continuous piecewise cubics:
\begin{equation*}
\mathbb{Y}_h = \{ y_h \in C^0(\overline{\Omega}; \mathbb{R}^3) : y_h|_T \in \mathcal{P}_3^{\mathrm{red}}(T) \text{ and } \nabla y_h \text{ is continuous at every node } z\in \mathcal{N}_h\}.
\end{equation*}
The superscript in $\mathcal{P}_3^{\mathrm{red}}(T)$ signifies that one of the ten degrees of freedom is eliminated by prescribing its value in terms of the other degrees of freedom, such that each function $y_h \in \mathbb{Y}_h$ is determined by its nodal values and nodal gradients.
To include clamped boundary conditions, we define the space
\begin{equation*}
\mathbb{Y}_{h,y_D,\Phi_D} = \{ y_h \in \mathbb{Y}_h : y_h(z) = y_D(z) \text{ and } \nabla y_h(z) = \Phi_D(z) \text{ for all } z\in \mathcal{N}_h\cap \Gamma_D\}.
\end{equation*}

A key feature of the DKT element is the existence of a discrete gradient operator $\nabla_h$ which approximates gradients of functions from $\mathbb{Y}_h$ in the space of continuous piecewise quadratic polynomials, see~\cite{braess2007finite} for its construction.
We may exploit this approximation and define the discrete energy via
\begin{equation*}
E_{\mathrm{bend}, h}[y_h] = \int_\Omega\frac{1}{2} |\nabla \nabla_h y_h|^2 - f\cdot y_h \; dx.
\end{equation*}
Following the approach for harmonic maps with a nodal imposition of the constraint, the discrete admissible set is given by
\begin{equation*}
\mathcal{A}_{h, y_D, \Phi_D} = \{ y_h \in \mathbb{Y}_{h,y_D,\Phi_D} : \nabla y_h(z)^\top \nabla y_h(z) = I_{2\times 2} \text{ for all } z\in \mathcal{N}_h \}.
\end{equation*}
Denoting the discrete homogeneous space by
\begin{equation*}
\mathbb{Y}_{h,D} = \{ y_h \in \mathbb{Y}_h : y_h(z) = 0 \text{ and } \nabla y_h(z) = 0 \text{ for all } z\in \mathcal{N}_h\cap \Gamma_D\},
\end{equation*}
the tangent space at $y_h \in \mathcal{A}_{h, y_D, \Phi_D}$ is obtained by linearizing the nodal isometry constraint and is given by
\begin{equation*}
T\mathcal{A}_{h, y_D, \Phi_D}(y_h) = \{ v_h \in \mathbb{Y}_{h,D} : \nabla y_h(z)^\top \nabla v_h(z) + \nabla v_h(z)^\top \nabla y_h(z) = 0 \text{ for all } z\in \mathcal{N}_h \}.
\end{equation*}
The approximation spaces are identical to those used in the classical projection-free scheme for nonlinear plate bending presented in~\cite{bartels2013bending}, which serves as the baseline in the experiment.

In order to design an unconstrained flow that avoids the saddle-point structure of the projection-free scheme, we construct a nodal tangent projection for DKT functions.
\begin{definition}[tangent projection operator]\label{def:tangent-proj-dkt}
For $v_h \in \mathbb{Y}_h$, we define $w_h := \Pi_{h, y_h} v_h$ on each node $z_i \in \mathcal{N}_h$ by
\begin{align*}
w_h(z_i) = v_h(z_i),
\qquad
\nabla w_h(z_i) = \mathop{\mathrm{argmin}}_{\substack{B \in \mathbb{R}^{3\times 2}\\ \nabla y_h(z_i)^\top B + B^\top \nabla y_h(z_i) = 0}} |B - \nabla v_h(z_i)|^2.
\end{align*}
\end{definition}
Note that the constrained least-squares problem above can be solved at low cost by inverting a single $3\times 3$ system at each node.

With the tangent projection in hand, Alg.~\ref{alg:unconstrained-flow} carries over to the plate bending problem via natural substitutions: the harmonic-map admissible set $\mathcal{A}_{h,\delta}$ is replaced by $\mathcal{A}_{h, y_D, \Phi_D}$, the tangent projection $P_{h, u_h}$ by $\Pi_{h, y_h}$, and both the $H^1$-inner product $(\nabla u_h, \nabla w_h)$ and the gradient flow metric $(u_h, w_h)_*$ by the discrete $H^2$-inner product $(\nabla \nabla_h y_h, \nabla \nabla_h w_h)$; the body force enters the right-hand side as the additional contribution $(f, w_h)$.
Since the gradient flow metric coincides with the energy norm, we only consider the unstabilized case $\gamma = 0$ in the following.
Each step of the resulting algorithm requires the solution of a single symmetric positive definite system whose matrix is independent of $\tau$ and can therefore be factorized once and reused throughout the iteration.

\subsection{Computational example}\label{subsec:square-plate}
We consider a square plate of side length~$4$, clamped along two adjacent sides and subjected to a constant vertical body force; see Figure~\ref{fig:square-plate}. To this end, let $\Omega = (0,4) \times (0,4)$ with clamped boundary $\Gamma_D = \bigl(\{0\} \times [0,4]\bigr) \cup \bigl([0,4] \times \{0\}\bigr)$, boundary data $y_D(x) = (x_1, x_2, 0)^\top$ and $\Phi_D = [I_2,\,0]^\top$, and body force $f = (0, 0, 0.5)^\top$. The initial deformation~$y_h^0$ is the nodal interpolant of the flat configuration~$y_D$, which satisfies the discrete isometry constraint exactly. We compare Alg.~\ref{alg:unconstrained-flow} adapted to plate bending against the classical projection-free scheme of Section~\ref{sec:intro-proj-free} in its plate-bending variant~\cite{bartels2013bending}.
In both schemes the discrete gradient flow is defined using the discrete $H^2$-inner product, i.\,e., $(y_h, v_h)_* = (\nabla \nabla_h y_h, \nabla \nabla_h v_h)$ and the adaptive time step parameter in Alg.~\ref{alg:unconstrained-flow} is chosen as $\alpha = 0.9$.

\begin{figure}
\centering
\includegraphics[width=0.65\textwidth]{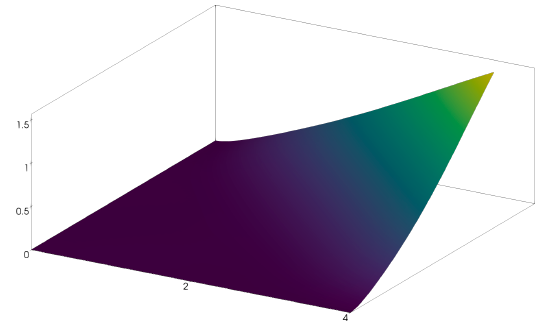}
\caption{Square plate of side length $4$ clamped along the left and bottom edges, subject to a constant vertical load.}
\label{fig:square-plate}
\end{figure}

Table~\ref{tab:plate-bending3} compares the two schemes for several values of the maximum step size~$\tau_{\max}$ at two mesh refinement levels. We report the total iteration counts $N$ and computing times required to reach the stopping criterion $\|d_t y_h^k\|_* \leq 10^{-3}$ for the projection-free scheme and $\|v_h^{k+1}\|_* \leq 10^{-3}$ for Alg.~\ref{alg:unconstrained-flow}, the final discrete energies $E_{\mathrm{bend},h}[y_h^\infty]$, and the discrete isometry errors $\delta_1[y_h] = \|\mathcal{I}_h(\nabla y_h^\top \nabla y_h - I_2)\|_{L^1(\Omega)}$ together with their experimental orders of convergence. For Alg.~\ref{alg:unconstrained-flow} we additionally report the smallest step size $\tau_{\min}$ taken by the adaptive controller over the whole run, and the number $N_{\mathrm{full}}$ of accepted steps for which the maximum step size $\tau=\tau_{\max}$ was used. For both schemes we observe linear convergence of the isometry defect~$\delta_1[y_h^\infty]$ with respect to~$\tau_{\max}$, which is in agreement with the known estimate for the projection-free scheme~\cite{bartels2013bending}. We find that the unconstrained scheme results in smaller constraint violations while being at least $11$ to $13$ times faster than the projection-free reference for smaller values of $\tau$. The relative gains are expected to grow when the number of degrees of freedom is increased further.

\begin{table}[h]
\centering
\caption{Comparison of Alg.~\ref{alg:unconstrained-flow} adapted to plate bending and the classical projection-free scheme on the square plate problem of Section~\ref{subsec:square-plate} at two grid refinement levels. Projection-free runs with $\tau < 2.441 \times 10^{-4}$ on the finer grid were omitted due to compute time.}
\label{tab:plate-bending3}
\begin{tabular}{rcrrrcccc}
\toprule
\#DOF & $\tau_{\max}$ & $N$ & $N_{\mathrm{full}}$ & time\,(s) & $\delta_1[y_h^\infty]$ & $\mathrm{EOC}_{\delta_1[y_h^\infty]}$ & $\tau_{\min}$ & $E_{\mathrm{bend},h}[y_h^\infty]$ \\
\hline
\multicolumn{9}{c}{Alg.~\ref{alg:unconstrained-flow} adapted to plate bending ($\alpha=0.9$)} \\
\hline
9216 & 7.812e-03 &  4101 &     9 &  44 & 1.405e-02 &    ---    & 1.808e-04 & -1.584e+00 \\
     & 3.906e-03 &  5995 &  1402 &  63 & 8.062e-03 & 8.010e-01 & 2.219e-04 & -1.580e+00 \\
     & 1.953e-03 &  7100 &  3907 &  63 & 4.530e-03 & 8.314e-01 & 4.149e-04 & -1.578e+00 \\
     & 9.766e-04 & 10635 & 10635 &  70 & 2.649e-03 & 7.743e-01 & 9.766e-04 & -1.576e+00 \\
     & 4.883e-04 & 21268 & 21268 & 140 & 1.324e-03 & 1.001e+00 & 4.883e-04 & -1.575e+00 \\
     & 2.441e-04 & 42534 & 42534 & 279 & 6.616e-04 & 1.000e+00 & 2.441e-04 & -1.575e+00 \\
\hline
\multicolumn{9}{c}{classical projection-free scheme} \\
\hline
9216 & 7.812e-03 &  1442 & --- &  202 & 1.456e-02 &    ---    & --- & -1.584e+00 \\
     & 3.906e-03 &  2878 & --- &  403 & 7.245e-03 & 1.007e+00 & --- & -1.579e+00 \\
     & 1.953e-03 &  5751 & --- &  803 & 3.614e-03 & 1.003e+00 & --- & -1.577e+00 \\
     & 9.766e-04 & 11497 & --- & 1604 & 1.805e-03 & 1.002e+00 & --- & -1.575e+00 \\
     & 4.883e-04 & 22988 & --- & 3204 & 9.019e-04 & 1.001e+00 & --- & -1.575e+00 \\
     & 2.441e-04 & 45971 & --- & 6394 & 4.508e-04 & 1.000e+00 & --- & -1.575e+00 \\
\hline
\multicolumn{9}{c}{Alg.~\ref{alg:unconstrained-flow} adapted to plate bending ($\alpha=0.9$)} \\
\hline
36864 & 1.953e-03 &  16987 &   1990 &  980 & 2.759e-03 &    ---    & 1.746e-04 & -1.242e+00 \\
      & 9.766e-04 &  22624 &   4949 & 1312 & 1.447e-03 & 9.305e-01 & 5.474e-05 & -1.241e+00 \\
      & 4.883e-04 &  26282 &  14112 & 1372 & 8.649e-04 & 7.428e-01 & 1.024e-04 & -1.241e+00 \\
      & 2.441e-04 &  39130 &  39130 & 1502 & 4.962e-04 & 8.016e-01 & 2.441e-04 & -1.240e+00 \\
      & 1.221e-04 &  78247 &  78247 & 3000 & 2.480e-04 & 1.000e+00 & 1.221e-04 & -1.240e+00 \\
      & 6.104e-05 & 156481 & 156481 & 6102 & 1.240e-04 & 1.000e+00 & 6.104e-05 & -1.240e+00 \\
\hline
\multicolumn{9}{c}{classical projection-free scheme} \\
\hline
36864 & 1.953e-03 & 5306  & --- &  4959 & 2.569e-03 &    ---    & --- & -1.242e+00 \\
      & 9.766e-04 & 10602 & --- &  9891 & 1.282e-03 & 1.002e+00 & --- & -1.241e+00 \\
      & 4.883e-04 & 21193 & --- & 19750 & 6.407e-04 & 1.001e+00 & --- & -1.240e+00 \\
      & 2.441e-04 & 42375 & --- & 40080 & 3.203e-04 & 1.001e+00 & --- & -1.240e+00 \\
\bottomrule
\end{tabular}
\end{table}

\bibliographystyle{plain}
\bibliography{references}

\end{document}

%% file: data/blowup_level0_combined.tex
\begin{tabular}{lllllll}
\toprule
 $\tau_{\mathrm{max}}/\tau$ & $t_s$ & $\Vert \mathcal{I}_h|u_h^\infty|^2 - 1 \Vert_{L^1}$ & $\Vert \mathcal{I}_h|u_h^\infty|^2 - 1 \Vert_{L^\infty}$ & $N$ & $E[u_h^\infty]$ & wall time (s) \\
\hline
\multicolumn{7}{c}{classical projection-free scheme} \\
\hline
0.007812 & 0.3672 & 3.4548e-01 & 2.0986e+00 & 64 & 8.4048 & 20.8645 \\
0.003906 & 0.3438 & 1.8875e-01 & 2.4083e+00 & 128 & 8.1314 & 42.2095 \\
0.001953 & 0.3125 & 1.0310e-01 & 2.5778e+00 & 256 & 7.7882 & 83.3932 \\
\hline
\multicolumn{7}{c}{Alg.~\ref{alg:unconstrained-flow-heat} with constant step sizes} \\
\hline
0.007812 & 0.3906 & 3.3446e-01 & 4.3342e+00 & 64 & 9.5792 & 3.6242 \\
0.003906 & 0.3398 & 1.8904e-01 & 3.5204e+00 & 128 & 8.5977 & 6.4041 \\
0.001953 & 0.3105 & 1.0350e-01 & 2.7107e+00 & 256 & 7.8534 & 12.0788 \\
\hline
\multicolumn{7}{c}{Alg.~\ref{alg:unconstrained-flow-heat} with $\alpha = 0.5$} \\
\hline
0.007812 & 0.3645 & 3.0235e-01 & 1.5182e+00 & 77 & 7.6518 & 37.1895 \\
0.003906 & 0.3293 & 1.8053e-01 & 1.4507e+00 & 138 & 7.2923 & 13.9539 \\
0.001953 & 0.3075 & 9.8730e-02 & 1.3859e+00 & 264 & 7.0675 & 19.8836 \\
\hline
\multicolumn{7}{c}{Alg.~\ref{alg:unconstrained-flow-heat} with $\alpha = 0.9$} \\
\hline
0.007812 & 0.2838 & 4.6441e-02 & 5.8712e-02 & 692 & 6.5401 & 41.3792 \\
0.003906 & 0.2824 & 3.2924e-02 & 5.8514e-02 & 918 & 6.5047 & 51.3656 \\
0.001953 & 0.2803 & 1.8674e-02 & 5.8142e-02 & 1499 & 6.4690 & 83.1647 \\
\bottomrule
\end{tabular}

%% file: references.bib
@article{akrivis2021higher,
  title={Higher-order linearly implicit full discretization of the {L}andau--{L}ifshitz--{G}ilbert equation},
  author={Akrivis, Georgios and Feischl, Michael and Kov{\'a}cs, Bal{\'a}zs and Lubich, Christian},
  journal={Mathematics of Computation},
  volume={90},
  number={329},
  pages={995--1038},
  year={2021}
}

@article{bartels2022quasi,
  title={Quasi-optimal error estimates for the approximation of stable harmonic maps},
  author={Bartels, S{\"o}ren and Palus, Christian and Wang, Zhangxian},
  journal={arXiv preprint arXiv:2209.11985},
  year={2022}
}

@article{bartels2005stability,
  title={Stability and convergence of finite-element approximation schemes for harmonic maps},
  author={Bartels, S{\"o}ren},
  journal={SIAM Journal on Numerical Analysis},
  volume={43},
  number={1},
  pages={220--238},
  year={2005},
  publisher={SIAM}
}

@article {bartels2013bending,
    author = {Bartels, S{\"o}ren},
     title = {Finite element approximation of large bending isometries},
   journal = {Numer. Math.},
  fjournal = {Numerische Mathematik},
    volume = {124},
      year = {2013},
    number = {3},
     pages = {415--440},
      issn = {0029-599X,0945-3245},
   mrclass = {65N30 (74K20 74S05)},
  mrnumber = {3066035},
       doi = {10.1007/s00211-013-0519-7},
       url = {https://doi.org/10.1007/s00211-013-0519-7}
}

@article{bartels2016projection,
  title={Projection-free approximation of geometrically constrained partial differential equations},
  author={Bartels, S{\"o}ren},
  journal={Mathematics of Computation},
  volume={85},
  number={299},
  pages={1033--1049},
  year={2016}
}

@article{barrett2007convergent,
  title={A convergent and constraint-preserving finite element method for the p-harmonic flow into spheres},
  author={Barrett, John W and Bartels, S{\"o}ren and Feng, Xiaobing and Prohl, Andreas},
  journal={SIAM Journal on Numerical Analysis},
  volume={45},
  number={3},
  pages={905--927},
  year={2007},
  publisher={SIAM}
}

@book{brenner2008mathematical,
  title={The mathematical theory of finite element methods},
  author={Brenner, Susanne C and Scott, L Ridgway},
  year={2008},
  publisher={Springer}
}

@book{bartels2015numerical,
  title={Numerical methods for nonlinear partial differential equations},
  author={Bartels, S{\"o}ren},
  volume={47},
  year={2015},
  publisher={Springer}
}

@article{alouges1997new,
  title={A new algorithm for computing liquid crystal stable configurations: the harmonic mapping case},
  author={Alouges, Fran{\c{c}}ois},
  journal={SIAM Journal on Numerical Analysis},
  volume={34},
  number={5},
  pages={1708--1726},
  year={1997},
  publisher={SIAM}
}

@article{xia2021augmented,
  title={Augmented {L}agrangian preconditioners for the {O}seen--{F}rank model of nematic and cholesteric liquid crystals},
  author={Xia, Jingmin and Farrell, Patrick E and Wechsung, Florian},
  journal={BIT Numerical Mathematics},
  volume={61},
  number={2},
  pages={607--644},
  year={2021},
  publisher={Springer}
}

@article{frank1958liquid,
  title={I. Liquid crystals. On the theory of liquid crystals},
  author={Frank, Frederick C},
  journal={Discussions of the Faraday Society},
  volume={25},
  pages={19--28},
  year={1958},
  publisher={Royal Society of Chemistry}
}

@article{friesecke2002theorem,
  title={A theorem on geometric rigidity and the derivation of nonlinear plate theory from three-dimensional elasticity},
  author={Friesecke, Gero and James, Richard D and M{\"u}ller, Stefan},
  journal={Communications on Pure and Applied Mathematics: A Journal Issued by the Courant Institute of Mathematical Sciences},
  volume={55},
  number={11},
  pages={1461--1506},
  year={2002},
  publisher={Wiley Online Library}
}

@article{gilbert2004phenomenological,
  title={A phenomenological theory of damping in ferromagnetic materials},
  author={Gilbert, Thomas L},
  journal={IEEE Transactions on Magnetics},
  volume={40},
  number={6},
  pages={3443--3449},
  year={2004},
  publisher={IEEE}
}

@article{bartels2024error,
  title={Error analysis for the numerical approximation of the harmonic map heat flow with nodal constraints},
  author={Bartels, S{\"o}ren and Kov{\'a}cs, Bal{\'a}zs and Wang, Zhangxian},
  journal={IMA Journal of Numerical Analysis},
  volume={44},
  number={2},
  pages={633--653},
  year={2024},
  publisher={Oxford University Press}
}

@article{bouck2024projection,
  title={Projection-Free Method for the Full {F}rank-{O}seen Model of Liquid Crystals},
  author={Bouck, Lucas and Nochetto, Ricardo H},
  journal={arXiv preprint arXiv:2405.03145},
  year={2024}
}

@article{nochetto2022gamma,
  title={Gamma-convergent projection-free finite element methods for nematic liquid crystals: The {E}ricksen model},
  author={Nochetto, Ricardo H and Ruggeri, Michele and Yang, Shuo},
  journal={SIAM Journal on Numerical Analysis},
  volume={60},
  number={2},
  pages={856--887},
  year={2022},
  publisher={SIAM}
}

@article{dong2025accelerated,
  title={Accelerated gradient flows for large bending deformations of nonlinear plates},
  author={Dong, Guozhi and Guo, Hailong and Yang, Shuo},
  journal={SIAM Journal on Scientific Computing},
  volume={47},
  number={5},
  pages={A2481--A2505},
  year={2025},
  publisher={SIAM}
}

@article{bartels2022stable,
  title={Stable gradient flow discretizations for simulating bilayer plate bending with isometry and obstacle constraints},
  author={Bartels, S{\"o}ren and Palus, Christian},
  journal={IMA Journal of Numerical Analysis},
  volume={42},
  number={3},
  pages={1903--1928},
  year={2022},
  publisher={Oxford University Press}
}

@article{hu2009saddle,
  title={A saddle point approach to the computation of harmonic maps},
  author={Hu, Qiya and Tai, Xue-Cheng and Winther, Ragnar},
  journal={SIAM Journal on Numerical Analysis},
  volume={47},
  number={2},
  pages={1500--1523},
  year={2009},
  publisher={SIAM}
}

@article{kraus2019iterative,
  title={Iterative solution and preconditioning for the tangent plane scheme in computational micromagnetics},
  author={Kraus, Johannes and Pfeiler, Carl-Martin and Praetorius, Dirk and Ruggeri, Michele and Stiftner, Bernhard},
  journal={Journal of Computational Physics},
  volume={398},
  pages={108866},
  year={2019},
  publisher={Elsevier}
}

@article{lopez2011frustrated,
  title={Frustrated nematic order in spherical geometries},
  author={Lopez-Leon, Teresa and Koning, Vinzenz and Devaiah, KBS and Vitelli, Vincenzo and Fernandez-Nieves, Alberto},
  journal={Nature Physics},
  volume={7},
  number={5},
  pages={391--394},
  year={2011},
  publisher={Nature Publishing Group UK London}
}

@article{fernandez2007novel,
  title={Novel defect structures in nematic liquid crystal shells},
  author={Fern{\'a}ndez-Nieves, Alberto and Vitelli, Vincenzo and Utada, Andrew S and Link, Darren R and M{\'a}rquez, Manuel and Nelson, David R and Weitz, David A},
  journal={Physical Review Letters},
  volume={99},
  number={15},
  pages={157801},
  year={2007},
  publisher={APS}
}

@article{chang1992finite,
  title={Finite-time blow-up of the heat flow of harmonic maps from surfaces},
  author={Chang, Kung-Ching and Ding, Wei Yue and Ye, Rugang},
  journal={Journal of Differential Geometry},
  volume={36},
  number={2},
  pages={507--515},
  year={1992}
}

@book{braess2007finite,
  title     = {Finite Elements: Theory, Fast Solvers, and Applications in Solid Mechanics},
  author    = {Braess, Dietrich},
  year      = {2007},
  edition   = {3},
  publisher = {Cambridge University Press},
  address   = {Cambridge},
  doi       = {10.1017/CBO9780511618635},
  isbn      = {978-0-521-70518-9}
}

@phdthesis{palus2024finite,
  title       = {Finite element simulation of a nonlinear bending model for nematic
                 liquid crystal elastomer plates and related geometrically
                 constrained problems},
  author      = {Palus, Christian},
  year        = {2024},
  school      = {Albert-Ludwigs-Universit{\"a}t Freiburg},
  type        = {Ph{D} dissertation},
  doi         = {10.6094/UNIFR/259095},
  url         = {https://doi.org/10.6094/UNIFR/259095}
}

@techreport{schoberl2014c++,
  title={C++ 11 implementation of finite elements in {NGS}olve},
  author={Sch{\"o}berl, Joachim},
  institution={Institute for Analysis and Scientific Computing, Vienna University of Technology},
  number={30},
  year={2014}
}
